\documentclass[11pt]{article}
\usepackage[utf8]{inputenc}
\usepackage[T1]{fontenc}
\usepackage{lmodern}
\usepackage[normalem]{ulem}
\usepackage[makeroom]{cancel}
\usepackage[english]{babel}

\hsize \textwidth
\advance \hsize by -\marginparwidth
\evensidemargin \oddsidemargin
\usepackage{pdfsync}

\usepackage{amsthm, amssymb, amsmath, amsfonts, mathrsfs}

\usepackage{mathtools}
\usepackage[ocgcolorlinks, colorlinks=true, pdfstartview=FitV, linkcolor=blue, citecolor=blue, urlcolor=blue, pagebackref=false]{hyperref}

\numberwithin{equation}{section}

-\marginparwidth \oddsidemargin -4mm \evensidemargin \oddsidemargin
\newcommand{\pdr}[2]{\frac{\partial{#1}}{\partial{#2}}}

\newcommand{\Rm}{{\mathbb R}}
\newcommand{\eps}{\varepsilon}
\newcommand{\commentout}[1]{}
\newcommand{\E}{{\mathbb E}}

\newcommand{\bbE}{\mathbb{E}}
\newcommand{\bbP}{\mathbb P}

\newcommand{\farc}{\frac}
\newcommand{\la}{\lambda}

\newcommand{\Sd}{{\cal S}_{\rm div}(\bbR^d;\bbR^d)}

\newtheorem{theorem}{Theorem}[section]
\newtheorem{lemma}[theorem]{Lemma}

\newtheorem{corollary}[theorem]{Corollary}
\newtheorem{proposition}[theorem]{Proposition}

\theoremstyle{remark}

\newcommand{\bes}{\begin{displaymath}}
\newcommand{\ees}{\end{displaymath}}
\newcommand{\be}{\begin{equation}}
\newcommand{\ee}{\end{equation}}
\newcommand{\ba}{\begin{eqnarray}}
\newcommand{\ea}{\end{eqnarray}}
\newcommand{\bas}{\begin{eqnarray*}}
\newcommand{\eas}{\end{eqnarray*}}

\newcommand{\B}{{\@Bbb B}}
\newcommand{\C}{{\@Bbb C}}

\newcommand{\F}{{\@Bbb F}}
\renewcommand{\P}{{\@Bbb P}}

\newcommand{\Q}{{\@Bbb Q}}
\newcommand{\bQ}{{\@Bbb Q}}
\newcommand{\N}{{\@Bbb N}}
\newcommand{\R}{{\@Bbb R}}

\newcommand{\W}{{\@Bbb W}}

\newcommand{\al}{\alpha}
\newcommand{\si}{\sigma}

\newcommand{\Om}{\Omega}

\newcommand{\ep}{\varepsilon}

\newcommand{\cA}{\@s A}
\newcommand{\cB}{\@s B}
\newcommand{\cC}{\@s C}
\newcommand{\cD}{\@s D}
\newcommand{\cE}{\@s E}
\newcommand{\cF}{\@s F}
\newcommand{\cG}{\@s G}
\newcommand{\cH}{\@s H}
\newcommand{\cI}{\@s I}
\newcommand{\cJ}{\@s J}
\newcommand{\cK}{\@s K}
\newcommand{\cL}{\@s L}
\newcommand{\cN}{\@s N}
\newcommand{\cM}{\@s M}
\newcommand{\cO}{\@s O}
\newcommand{\cP}{\@s P}
\newcommand{\cR}{\@s R}
\newcommand{\cS}{\@s S}
\newcommand{\cT}{\@s T}
\newcommand{\cV}{\@s V}
\newcommand{\cW}{\@s W}
\newcommand{\cX}{\@s X}
\newcommand{\cY}{\@s Y}
\newcommand{\cZ}{\@s Z}

\newcommand{\bma}{\@bm a}
\newcommand{\bmb}{\@bm b}
\newcommand{\bmc}{\@bm c}
\newcommand{\bmd}{\@bm d}

\newcommand{\bme}{\@bm e}
\newcommand{\bmf}{\@bm f}
\newcommand{\bmg}{\@bm g}
\newcommand{\bmh}{\@bm h}
\newcommand{\bmi}{\@bm i}
\newcommand{\bmj}{\@bm j}
\newcommand{\bmk}{\@bm k}
\newcommand{\bml}{\@bm l}
\newcommand{\bmm}{\@bm m}
\newcommand{\bmn}{\@bm n}
\newcommand{\bmo}{\@bm o}
\newcommand{\bmp}{\@bm p}
\newcommand{\bmq}{\@bm q}
\newcommand{\bmr}{\@bm r}
\newcommand{\bms}{\@bm s}
\newcommand{\bmt}{\@bm t}
\newcommand{\bmu}{\@bm u}
\newcommand{\bmw}{\@bm w}
\newcommand{\bmv}{\@bm v}
\newcommand{\bmx}{\@bm x}
\newcommand{\bx}{\@bm x}
\newcommand{\bmy}{\@bm y}
\newcommand{\bmz}{\@bm z}

\newcommand{\by}{\@bm y}
\newcommand{\bmzero}{\@bm 0}

\newcommand{\ga}{\gamma}

\newcommand{\gA}{\@g A}
\newcommand{\gD}{\@g D}
\newcommand{\gJ}{\@g J}
\newcommand{\gF}{\@g F}
\newcommand{\gM}{\@g M}
\newcommand{\gR}{\@g R}
\newcommand{\bbR}{\mathbb R}

\author{Yu Gu\thanks{Department of Mathematics, Carnegie Mellon University, Pittsburgh, PA 15213, USA (yug2@andrew.cmu.edu)} \and Tomasz Komorowski\thanks{Institute of Mathematics, Polish Academy of Sciences, ul. \'{S}niadeckich 8, 00-656 Warsaw, Poland (komorow@hektor.umcs.lublin.pl)} \and Lenya Ryzhik\thanks{Department of Mathematics, Building 380, Stanford University, Stanford, CA, 94305, USA (ryzhik@stanford.edu)}}
\date{}

\title{Fluctuations of random semi-linear advection equations}

 \date{\today}

\begin{document}

\maketitle

\begin{abstract}
We consider a semi-linear advection equation driven by a highly-oscillatory space-time Gaussian random field, with the randomness affecting both the drift and the nonlinearity. In the linear setting, classical results show that the characteristics converge in distribution to a homogenized Brownian motion, hence the point-wise law of the solution is  close to a functional of the Brownian motion.   Our main result is 
that the nonlinearity plays the role of a \emph{random diffeomorphism}, and the point-wise limiting distribution is obtained by applying the diffeomorphism to the limit in the linear setting.\end{abstract}

\section{Introduction}

In this paper, we consider solutions to the semi-linear advection equations with rapidly oscillating
random coefficients, 
of the form
\begin{equation}\label{e.maineq}
\begin{aligned}
&\partial_t u_\eps(t,x)+\frac{1}{\eps}V\Big(\frac{t}{\eps^2},\frac{x}{\eps}\Big)\cdot\nabla_x
u_\eps(t,x)= \frac{1}{\eps^\alpha}f\Big(t,x,u_\eps(t,x),V\Big(\frac{t}{\eps^2},\frac{x}{\eps}+\cdot\Big)\Big),  \\
&u(T,x)=u_0(x), \quad  t<T, x\in\bbR^d.
\end{aligned}
\end{equation}
Here, $V(t,x)$ is a zero-mean, incompressible, stationary Gaussian,
vector-valued random field, and
the nonlinear term $f$ depends on both $u_\eps$ and $V$. The parameter
$\alpha\geq0$ is to be chosen so that the nonlinearity plays a
non-trivial role as $\eps\to0$. {The linear problem
  with $f\equiv 0$ corresponds to the {\em passive scalar} model  that
  describes a particle drifting in a time-dependent, incompressible
  random  environment
  and has applications in both turbulent diffusion and stochastic homogenization, see e.g. \cite{kramer-majda,warhaft,sreenivasan} and the references therein.
The model has been extensively studied, both in the mathematics and physics  literature, under various assumptions
on the advection $V(t,x)$. 
A typical result shows that the underlying characteristics, i.e., the trajectory of the particle, converge to
a diffusion, see e.g. \cite{bp,cf,carmona-xu,fk,kom,kps}. 
The problem may also exhibit
a memory effect if the space-time correlations of $V(t,x)$ decay sufficiently
slowly so that the corresponding trajectory process converges to a non-Markovian
limit, see \cite{AM,fako,kola,koryz}. }

{
Homogenization problems for   quasilinear, stochastic Hamilton-Jacobi type  equations,
with or without the presence of a viscous term, were
extensively studied in the periodic, almost periodic and ergodic settings, starting with the work of \cite{LPV}, see
also \cite{cs,tr,s,k-v,k,as} and the references therein. These problems usually
involve the hyperbolic scaling, i.e., the time and spatial variables are
scaled as $t/\eps,x/\eps$, compared with the diffusive scaling
appearing in \eqref{e.maineq}. Concerning the diffusive scaling, the
question of homogenization for   some classes  of semi-linear (or even
of quasi-linear) 
  parabolic equations with periodic, or ergodc coefficients, using backward stochastic differential equation
  techniques, has been considered in  \cite{kom,delarue,pardoux,rhodes1,rhodes2}.}

 In this paper, we stay in the regime where $V$ decorrelates fast, and our goal is to understand the
interaction between the randomness and the nonlinearity, and  
the asymptotic behavior of $u_\eps$, as well as 
the multi-point statistics 
$u_\eps(t,x_1),\ldots,u_\eps(t,x_N)$ for any number of points
$(t,x_1),\ldots,(t,x_N)$.

As we have mentioned, when $f\equiv0$, the equation~\eqref{e.maineq}
is a classical problem  of
a passive scalar in an evolving  random environment, 
and the solution can be expressed as
\begin{equation}\label{e.character}
u_\eps(t,x)=u_0(X^{t,x}_\eps(T)).
\end{equation}
Here, $X^{t,x}_\eps(\cdot)$
is the  characteristic of \eqref{e.maineq}
starting from $(t,x)$:
\begin{equation}\label{e.characteristics}
\frac{d}{ds}X^{t,x}_\eps(s)=\frac{1}{\eps}V\Big(\frac{s}{\eps^2},\frac{X^{t,x}_\eps(s)}{\eps}\Big),\, s>t; \quad X^{t,x}_\eps(t)=x.
\end{equation}
It was shown in, e.g., \cite{carmona-xu,fk,koralov} that 
the process $(X_\eps^{t,x}(s))_{s\ge t}$ converges in law to 
$(x+\beta_{s-t})_{s\ge t}$. Here,~$(\beta_t)_{t\ge0}$ is
a Brownian motion
with a covariance matrix that can be computed through the statistics of
$V$, see \eqref{apq} below. As a result, for fixed $(t,x)$, $u_\eps(t,x)$ converges in
distribution to~$u_0(x+\beta_{T-t})$. For two different starting
points $x_1\neq x_2$, the trajectories $X^{t,x_1}_\eps$ and
$X^{t,x_2}_\eps$ experience random environments that are typically at
distances of order
$O(1/\eps)$ away from each other on the microscopic spatial scale. 
As a result, the two trajectories
become nearly independent, when $\eps\to0$, provided that the velocity field
$V(t,x)$ decorrelates fast in space. This happens even if
the spatial realizations of $V$ 
are analytic, which precludes the spatial mixing of the
field. Similarly, for an
arbitrary number of initial starting points,  
the random vector~$\left(u_\eps(t,x_1),\ldots,u_\eps(t,x_N)\right)$ converges
in law to~$(u_0(x_1+\beta_{T-t}^{(1)}),\ldots,
  u_0(x_N+\beta_{T-t}^{(N)}))$, where
$(\beta_{t}^{(j)})_{t\ge0}$ are i.i.d. copies of the
effective Brownian motion, see Theorem~\ref{thm} below.
In
particular, the above result implies that, after averaging in space
(i.e. taking the weak spatial limit), the randomness averages out and
the limit becomes deterministic. More precisely we  have 
\begin{equation}
\label{011301}
\lim_{\eps\to0}\int_{\bbR^d} u_\eps(t,x) g(x)dx= \int_{\bbR^d} \E[u_\eps(t,x)] g(x)dx\quad\mbox{ for $g\in L^1(\bbR^d)$,}
\end{equation}
 see Corollary \ref{thm-c}.  {This remains  in sharp contrast with
  the parabolic setting  (see \cite{kom}),
  where both the point-wise limit and  the limit measured weakly in the 
  spatial variable are both deterministic.}

In the non-linear setting, when $f\neq0$, 
the solution along the characteristics is not constant but rather satisfies 
\begin{equation}\label{e.ch1}
u_\eps(s,X^{t,x}_\eps(s))+\frac{1}{\eps^{\alpha}}\int_s^T f\Big(\sigma,X^{t,x}_\eps(\sigma), u_\eps(\sigma,X^{t,x}_\eps(\sigma)), V(\frac{\sigma}{\eps^2},\frac{X^{t,x}_\eps(\sigma)}{\eps}+\cdot)\Big)d\sigma = u_0(X^{t,x}_\eps(T)), \quad s\in[t,T].
\end{equation}
If the nonlinearity has a non-zero mean $\bar{f}=\E[f]$, we can roughly treat it as 
deterministic to the leading order, in light of the averaging induced 
by the $V$ variable in~\eqref{e.ch1}. This leads to
the choice $\alpha=0$. Replacing~$f\to \bar{f}$, we obtain
from (\ref{e.ch1}):
\begin{equation}\label{e.inte1}
u_\eps(s,X^{t,x}_\eps(s))+\int_s^T \bar{f}\Big(\sigma,X^{t,x}_\eps(\sigma), u_\eps(\sigma,X^{t,x}_\eps(\sigma)\Big)d\sigma = u_0(X^{t,x}_\eps(T)), \quad s\in[t,T],
\end{equation}
a ``deterministic'' integral equation in time, driven by
the random charateristics. Since $X^{t,x}_\eps$ converges to the
effective Brownian motion, it is not hard to see from \eqref{e.inte1},
at least formally, that~$u_\eps(s,X^{t,x}_\eps(s))$ converges to
the
solution $\mathcal{U}(t,x)$ of an integral equation driven by the effective Brownian
motion. 
This
argument can be also extended to
arbitrary points $x_1,\ldots,x_N$, showing that 
random vectors~$(u_\eps(t,x_1),\ldots,u_\eps(t,x_N))$ converge
in law to $\left(\mathcal{U}^{(1)}(t,x_1),\ldots,
  \mathcal{U}^{(N)}(t,x_N)\right)$, where  $\mathcal{U}^{(j)}(t,x)$
correspond to solutions driven by independent copies of the
effective Brownian motion. If the fluctuation is measured
weakly-in-space, it can also be shown that \eqref{011301} holds. 
The precise statement of the results can be found in
Theorem~\ref{thm-semilin} and Corollary \ref{thm-semilin-c}.

When $\bar{f}=0$, the random effect of $f$
comes up in the next order, and 
the standard central limit scaling suggests the choice $\alpha=1$. Due to the interaction between the two random sources, $X^{t,x}_\eps$ and~$V$, the asymptotic behavior of the integral
\[
\frac{1}{\eps}\int_s^T f\Big(\sigma,X^{t,x}_\eps(\sigma), u_\eps(\sigma,X^{t,x}_\eps(\sigma)), V(\frac{\sigma}{\eps^2},\frac{X^{t,x}_\eps(\sigma)}{\eps}+\cdot)\Big)d\sigma
\]
that appears in (\ref{e.ch1}) is much more complicated than that in \eqref{e.inte1}. 
We will in this case obtain in the limit a ``random'' integral equation driving by the effective Brownian motion. 
This is the objective of Theorem~\ref{thm011110}.

Let us briefly describe the principal ingredients of the proofs of our main
results and organization of the paper. Section~\ref{sec1.1} contains the main results of this paper, and the
assumptions on the random advection $V(t,x)$. The analysis of the solutions of semi-linear
advection equations is based on the method of characteristics that translates the asymptotics of $u_\eps$ into the study 
of the random spatial trajectories, together with
the evolution  of $u_\eps$ along the characteristics, described by (\ref{e.ch1}),
together with the inverse of the corresponding flow map
coming from (\ref{e.ch1}). An important tool
in this  approach is the process that describes the random velocity $V$
along the spatial characteristics~--~the so-called environment process, see Section~\ref{s.corrector}. 
The main technical novelty in the analysis here is the approach to the analysis of the environment process.
It is shown in Section \ref{s.markov} that the Gaussian velocity fields, considered in the present paper, are actually
Markovian. Fields  of this type
appeared quite frequently throughout the literature, see e.g. \cite[Chapter 12]{KLO} and
the references therein. What is novel in our present approach,
compared with that of \cite{KLO},
is  the use of the respective Cameron-Martin space  in the description of the dynamics of the field,
see  Section \ref{s.markov}. It allows us to find a simple
semimartingale representation  of the dynamics,
see the stochastic differential equation \eqref{sde1}, which leads to
the It\^o formula \eqref{f}. This in turn allows us to find  the semimartingale
description of the environment process and the respective  It\^o 
formula, see Section \ref{s.corrector}. Using these tools we present
the proofs of our main results in Sections \ref{sec4.3} -- \ref{sec8.2}.

{\bf Acknowledgment.} YG is partially supported by the NSF
grant DMS-1613301/1807748 and the Center for Nonlinear Analysis of CMU, TK by the NCN grant 2016/23/B/ST1/00492 and 
LR by the NSF grants DMS-1311903 and DMS-1613603, and by ONR grant N00014-17-1-2145.

\section{Main results}

\label{sec1.1}

\subsection{Gaussian incompressible vector fields}

\label{sec2.1}
Let us first make precise our assumptions on the random field $V(t,x)=(V_1,\ldots,V_d)$.
It is a mean-zero, space-time stationary $d$-dimensional Gaussian random field, defined
on a probability space~$(\Om,{\cal F},\bbP$), with a covariance matrix  of the form
\begin{equation}
\label{vel}
R_{ij}(t,x)=\bbE\left[V_{i}(s+t,y+x)V_{j}(s,y)\right]=\int_{\bbR^d}e^{ix\cdot
  k}e^{-\al(k)|t|}\Gamma_{ij}(k)\si(k) dk,~~\hbox{ $i,j=1,\ldots d$. }
\end{equation}
The factor
\[
\Gamma(k):=[\Gamma_{ij}(k)],~~\Gamma_{ij}(k):=\delta_{i,j}-k_ik_j/|k|^2,~~i,j=1,\ldots,d,
\] 
ensures that the realizations of the field are almost surely incompressible:
$$
\nabla_x\cdot V(t,x)=\sum_{j=1}^d\partial_{x_j}V_j(t,x)\equiv 0,\quad (t,x)\in\bbR^{1+d},\,\mbox{a.s.}
$$ 
The temporal factor taking the form of $e^{-\alpha(k)|t|}$ plays an important role for our construction of the underlying Markovian dynamics. The non-negative functions~$\al(k)\ge 0$ and $\sigma(k)\ge 0$  are assumed to be even and continuous.
We also assume that $\sigma(k)$ is compactly supported: $\sigma(k)=0$ for $|k|\ge K_0$, and the spectral gap
$\alpha(k)$ is uniformly positive: 
\begin{equation}\label{nov1202}
0<\al_*\le \al(k)\le A_*, \  \ k\in\bbR^d.
\end{equation}
This property implies uniform  mixing in time of the velocity field, see \eqref{022801a} and \eqref{022801b} below.

In order to specify the function space where $V(t,x)$ takes its values, 
given $m_1,m_2\in\bbR$, let
$
{\cal E}_{m_1,m_2}
$ be the real Hilbert
space  
of vector-valued functions $w:\bbR^d\to\bbR^d$ with the norm
$$
\|w\|^2_{m_1,m_2}:=\int_{\bbR^d}\theta_{-m_2}(x)
[{\cal F}^{-1}(\theta_{m_1}\hat w)(x)]^2dx,~~\theta_{m}(x) :=(1+|x|^2)^{m/2}. 
$$
Here, the Fourier transform  and its inverse are  defined as
$$
\hat w(k)=[{\cal F}w](k):=\int_{\bbR^d}e^{-ik\cdot x}w(x)dx,~~
[{\cal F}^{-1}u](x):=\frac{1}{(2\pi)^d}\int_{\bbR^d}e^{ik\cdot x}\hat u(k)dk,\quad
w,u\in{\cal S}(\bbR^d).
$$
It is straightforward to check that the dual space  $
{\cal E}_{m_1,m_2}'$ to ${\cal E}_{m_1,m_2}$ is 
${\cal  E}_{-m_1,-m_2}$.
Note that the Dirac function $\delta(x)$
belongs to $
{\cal E}_{m_1,m_2}'
$, provided that 
$m_1>d$ and $m_2\in\bbR$.





Under the above assumptions,  for a fixed
$t\in\bbR$, the realizations of the components of
$V(t,\cdot)$ 
belong
a.s. to any  $
{\cal E}_{m_1,m_2}
$, with $m_1\in\bbR$ and $m_2>d$. Let ${\cal E}$ be the Hilbert space
consisting of vector fields
$w=(w_1,\ldots,w_d):\bbR^d\to\bbR^d$ whose components belong to $ 
{\cal E}_{m_1,m_2}
$ for some $m_1\ge1$ and $m_2>d$ satisfying $\nabla_x\cdot
w(x)\equiv0$, and let ${\cal B}({\cal E})$ be its Borel $\si$-algebra.
We denote by $\pi$ the law of $V(0,\cdot)$ (which coincides with the law
of $V(t,\cdot)$ for any $t\in\bbR$, due to stationarity)
over the space $({\cal E},{\cal B}({\cal E}))$.
  
%

\subsection{The linear case}

\label{sec7.2}

Let us first consider 
\begin{equation}
\label{advec-11}
\partial_t u_\eps(t,x)+\dfrac1\eps V\left(\dfrac{t}{\eps^2},\dfrac{x}{\eps}\right)\cdot\nabla_x u_\eps(t,x)=0,
~~u_\eps(T,x)=u_0(x),~~0\le t\le T,
\end{equation}
with a terminal condition $u_0$ that belongs to $C_0^\infty(\bbR^d)$. 
The
solution of \eqref{advec-11} is
\begin{equation}
\label{100805}
u_\eps(t,x)=u_0\big(X_\eps^{t,x}(T)\big),
\end{equation}
where we recall that $X_\eps^{t,x}(s)$  is the characteristic curve defined in \eqref{e.characteristics}
\begin{equation}
\begin{aligned}
\label{ode}
&
\frac{d X_\eps^{t,x}(s)}{ds}=\frac{1}{\eps}
V \Big(\frac{s}{\eps^2},\frac{X^{t,x}_\eps(s)}{\eps}\Big), ~~s>t,\\
 &X_\eps^{t,x}(t)=x.
\end{aligned}
\end{equation} 
It is well known, see \cite{carmona-xu,fk,koralov}, that under our assumptions on $V(t,x)$, 
the laws of~$\left(X_\eps^{t,x}(s)\right)_{s\ge t}$ converge, as $\eps\to 0$, to the law of 
$(x+\beta_{s-t})_{s\ge t}$. Here,  $\beta_t=(\beta_t^1,\ldots, \beta_t^d)$, 
is a $d$-dimensional Brownian motion with the covariance
\begin{equation}
\label{cov-B}
\bbE[\beta_t^p \beta_s^q]=a_{pq}(t\wedge s),\quad p,q=1,\ldots,d,\,t,s\ge0,
\end{equation}
and the effective diffusivity matrix $a_{pq}$ given by \eqref{apq} below.

The above implies, in particular, that for each $(t,x)$ fixed, $t\le T$, 
the random variables
$u_\eps(t,x)$, converge in law to a random variable
$u_0( x+\beta_{T-t} )$.
In addition, $\bar u(t,x):=\bbE\left[ u_0\left( x+\beta_{T-t}\right)\right]$
is the bounded solution of the  backward heat equation
\begin{equation}
\label{bheat-11}
\begin{aligned}
&\partial_t \bar u(t,x)+\frac12\sum_{p,q=1}^d a_{pq}\partial_{x_p,x_q}^2\bar
u(t,x)=0,\quad t\le T,
\\
&\bar u(T,x)=u_0(x).
\end{aligned}
\end{equation}
{For the multi-point statistics we have the following.
\begin{theorem}
\label{thm} For a given positive integer $N$, mutually distinct points
$x_1,\ldots,x_N\in\bbR^d$, and a fixed $t\le T$, the random vectors
$(u_\eps(t,x_1),\ldots,u_\eps(t,x_N))$ converge in law, as $\eps\to 0$, to 
$$
\left(u_0\left( x_1+\beta^{(1)}_{T-t}\right),\ldots,u_0\left( x_1+\beta^{(N)}_{T-t}\right) \right),
$$
where $(\beta^{(j)}_t)_{t\ge0}$, $j=1,\ldots,N$ are i.i.d. $d$-dimensional
Brownian
motions, {each one} with the covariance matrix given by \eqref{cov-B}.
\end{theorem}
Our result in the linear case allows, in particular,  to contrast the point-wise convergence 
of~$u_\eps(t,x)$ to a random limit, with the convergence of~$u_\eps(t,\cdot)$ in the weak
topology in $L^2(\bbR^d)$ to a deterministic limit. We use the notation 
\[
\langle
f,g\rangle:=\int_{\bbR^d} f(x)g(x) dx
\]
for $f\in L^p(\bbR^d)$, $g\in
L^{p'}(\bbR^d)$, with $1/p+1/p'=1$, $p\in[1,+\infty]$.
\begin{corollary}\label{c.deterlimit}
\label{thm-c} For a given $\varphi\in L^1(\bbR^d)$ and $t\le T$, the random variables
$$
\lim_{\eps\to0}\langle u_\eps(t),\varphi\rangle=\langle \bar u(t),\varphi\rangle \quad \mbox{in the $L^2(\Omega)$ sense}.
$$
\end{corollary}}
The proofs of the above results are presented in Section
\ref{sec4.3}. 

{\bf Remark.} The reason why we obtain a deterministic limit in Corollary~\ref{c.deterlimit} is due to the spatial averaging, which removes the local fluctuations after testing with $\varphi$. If one is interested in the local fluctuation, i.e., the fluctuations averaged on a ``small scale'' $\delta^{-d}\int_{\bbR^d}  \varphi(x/\delta)u_\eps(t,x)dx$ with $\delta\ll1$,  then the pointwise quantity $u_\eps(t,x)$ might be a more relevant object ($\delta\to0$) compared to the global fluctuation described by $\int u_\eps(t,x)\varphi(x)dx$ ($\delta=1$).

\bigskip

 Let us also comment that when $u_\eps(t,x)$ satisfies an advection-diffusion equation rather than an advection equation,
as in (\ref{advec-11}),
\begin{equation}
\label{advec-11a}
\begin{aligned}
&\partial_t u_\eps(t,x)+\dfrac1\eps V\left(\dfrac{t}{\eps^2},\dfrac{x}{\eps}\right)\cdot\nabla_x u_\eps(t,x)+\kappa\Delta_xu_\eps(t,x)=0,\\
&u_\eps(T,x)=u_0(x),
\end{aligned}
\end{equation}
with $\kappa>0$, one can prove, see \cite{kom}, that for any $t\le T$ both  
$u_\eps(t,x)$
and  $\langle
u_\eps(t),\varphi\rangle$ converge in probability to  deterministic limits $\bar u(t,x)$
and  $\langle
\bar u(t),\varphi\rangle$, respectively. In that case, 
$\bar u(t,x)$ is  the solution of the Cauchy problem for the backward heat equation
\begin{equation}
\label{bheat-11a}
\begin{aligned}
&\partial_t \bar u(t,x)+\frac12\sum\limits_{p,q=1}^d a_{pq}\partial_{x_p,x_q}^2\bar
u(t,x)+\kappa\Delta_x\bar u(t,x)=0,\quad t\le T,\\
&\bar u(T,x)=u_0(x).
\end{aligned}
\end{equation}
In other words, the diffusion term in (\ref{advec-11a}) provides enough extra averaging so that even the point-wise limit is deterministic.

\subsection{The semi-linear case}

Let $D_T:=[0,T]\times \bbR^{d+1}$ and
$C^{0,m}(D_T)$ be the
space of continuous functions $g(t,x,u)$ on $D_T$, that are
of the class $C^m$ in the $(x,u)$ variables for some non-negative integer
$m$:
$$
\|g\|_{C^{0,m}(D_T)}:=\sum_{|k|=0}^m\sup_{(t,x,u)\in D_T}
|D^k_{x,u}g(t,x,u)|.
$$
We consider semi-linear equations of the form  
\begin{equation}
\label{advec-11-semi-bis}
\begin{aligned}
&\partial_t u_\eps(t,x)+\dfrac1\eps
V\Big(\dfrac{t}{\eps^2},\dfrac{x}{\eps}\Big)\cdot\nabla_x
u_\eps(t,x)=  f_\eps(t,x,u_\eps(t,x)),\quad t< T,\\
&u_\eps(T,x)=u_0(x),
\end{aligned}
\end{equation}
with
\begin{equation}
\label{012806}
f_\eps(t,x,u):=f\Big(t,x,u,V\Big(\frac{t}{\eps^2},\frac{x}{\eps}+\cdot\Big)\Big),
\end{equation}
and $f(\cdot,w)\in C^{0,m}(D_T)$ for some $m>(d+1)/2$ and $\pi$-a.s. $w\in  {\cal E}$, and
\begin{equation}
\label{H1}
\mathop{\rm esssup}\limits_{w\in{\cal E}}\|f(\cdot,w)\|_{C^{0,m}(D_T)}<+\infty.
\end{equation}
Note that we omitted the dependence of $f_\eps$ on the random realization $\omega$ to simplify the notation. As  $f_\eps(t,x,u)$ is now random, the results will depend on whether it has 
a zero or non-zero mean
\begin{equation}
\label{021906}
\bar f(t,x,u):=\bbE f\left(t,x,u,V(0,\cdot)\right),\quad (t,x,u)\in D_T,
\end{equation} 
and we will consider these two cases separately. In the non-centered
case we have the following result, proved in Section \ref{sec9}.

{\begin{theorem}
\label{thm-semilin} Assume that $\bar f(t,x,u)\not\equiv 0$ in $D_T$.
Fix $(t,x)\in [0,T]\times\bbR^d$ and the realization of the Brownian motion $\beta_t$ with the covariance matrix \eqref{cov-B}, 
and let $\{\mathcal{U}(s;t,x)\}_{t\leq s\leq T}$ satisfy the integral equation 
\begin{equation}
\label{advec-homog-semi}
u_0\left(x+\beta_{T-t}\right)-\mathcal{U}(s;t,x)= \int_s^T\bar f\left(\si,x+\beta_{\si-t},\mathcal{U}(\si;t,x)\right)d\si,\quad t\le s\le T.
\end{equation}
Then  $u_\eps(t,x)$ converges in
law, as~$\eps\to 0$, 
to  $ \mathcal{U}(t,x):=\mathcal{U}(t;t,x)$. Moreover, for
any positive integer $N$, mutually distinct $x_1,\ldots,x_N\in \bbR^d$
and $t\le T$ the random vectors
$(u_\eps(t,x_1),\ldots,u_\eps(t,x_N))$ converge in law, as $\eps\to 0$, to 
$ (\mathcal{U}^{(1)}(t,x_1),\ldots, \mathcal{U}^{(N)}(t,x_N))$,
where $\mathcal{U}^{(1)},\ldots, \mathcal{U}^{(N)}$ correspond to the
solutions of \eqref{advec-homog-semi} with $\beta$ replaced by i.i.d. copies of $d$-dimensional
Brownian
motions $(\beta^{(j)}_t)_{t\ge0}$, $j=1,\ldots,N$,  whose  covariance
matrix is given by \eqref{cov-B}.
\end{theorem}
From the above result, we conclude an analogue of Corollary \ref{thm-c}:
\begin{corollary}
\label{thm-semilin-c}
Suppose that
$\varphi\in L^1(\bbR^d)$. Then
\begin{equation}
\label{020901}
\lim_{\eps\to0}\langle u_\eps(t),\varphi\rangle=\langle \bbE\,
\mathcal{U}(t),\varphi\rangle\quad \mbox{in the $L^2(\Omega)$ sense}.
\end{equation}  
\end{corollary}}

If
$\bar f(t,x,u)\equiv0$ for $(t,x,u)\in D_T$, the leading order effect of $f$
in (\ref{advec-homog-semi}) vanishes, so, to have the nonlinearity play a non-trivial role,
we consider instead of (\ref{advec-11-semi-bis}) the problem
\begin{equation}
\label{advec-11-semi2}
\begin{aligned}
&\partial_t u_\eps(t,x)+\dfrac1\eps
V\left(\dfrac{t}{\eps^2},\dfrac{x}{\eps}\right)\cdot\nabla_x
u_\eps(t,x)=  \dfrac{1}{\eps}f_\eps(t,x,u_\eps(t,x)),\quad t< T,
\\
&u_\eps(T,x)=u_0(x).
\end{aligned}
\end{equation}
 Here, $f_\eps$ is as in \eqref{012806}. We will, however, require slightly more
regularity on $f_\eps$. Let $C^{m}(D_T)$ be the
space of  continuous functions $g:D_T\times {\cal E}\to\bbR$, that  are
of the class $C^m$ in the $(s,x,u)$ variables for some non-negative integer
$m$:
$$
\|g\|_{C^{m}(D_T)}:=\sum_{|k|=0}^m\sup_{(s,y,u)\in D_T}|D^k_{s,y,u}g(s,y,u)|.
$$
We assume that $f(t,x,u,\omega)$  
is such that 
$f(\cdot,w)\in C^{m}(D_T)$ for some 
  $m>(d+1)/2$ and 
  $\pi$-a.s.~$w\in  {\cal E}$, and
  \begin{equation}
  \label{H10}
  \mathop{\rm esssup}\limits_{w\in{\cal E}}\|f(\cdot,w)\|_{C^{m}(D_T)}<+\infty.
\end{equation}

In order to state the result, let  $U^{t,x,u}_\eps(s)$ be the solution of  \eqref{advec-11-semi2} along the characteristics 
$X_\eps^{t,x}(s)$ given by (\ref{ode}), satisfying  $U^{t,x,u}_\eps(t)=u$. In other words, it is the solution of the equation
\begin{equation}
\label{011906aa}
U^{t,x,u}_\eps(s) =u+\frac{1}{\eps}\int_t^s f_\eps(\si)d\si,\quad t\le s \le T,
\end{equation}
where for $f:\bbR^{d+2}\times {\cal E}\to\bbR$ we simply write
$$
f_\eps(\sigma):=f\Big(\sigma,X_\eps^{t,x}(\sigma),U_\eps^{t,x,u}(\sigma),V\Big(\frac{\sigma}{\eps^2},\frac{X_\eps^{t,x}(\sigma)}{\eps}+\cdot\Big)\Big).
$$
For a fixed pair $(t,x)$, we define a random field 
\begin{equation}
{\mathfrak S}_{\eps}^{t,x}(s,u):=U_\eps^{t,x,u}(s), \    \   t\le s\le T, \  \ u\in\Rm.
\end{equation}
In order to define the limit of ${\mathfrak S}_{\eps}^{t,x}(s,u)$, let us introduce the solution of 
the following system of It\^o stochastic differential equations
\begin{equation}\label{020307}
\begin{aligned}
&
U^{t,x,u}(s)=u+\int_t^sb(\si,X^{t,x}(\si),U^{t,x,u}(\si))d\si+\sum_{j=0}^d\int_t^s\tilde c_{j}(\si,X^{t,x}(\si),U^{t,x,u}(\si))d\tilde\beta_j(\si),\\
&
X^{t,x}_j(s)=x_j+\sum_{k=1}^d\int_t^sS_{jk} d\tilde\beta_{k}(\si),\quad j=1,\ldots,d.
\end{aligned}
\end{equation}
Here, the coefficients $b(s,x,u)$ and $\tilde c_{j}(s,x,u)$ are defined in Section~\ref{sec7.1} below, $\tilde\beta_j(\sigma)$, 
$j=0,\ldots,d$, are
i.i.d. standard Brownian motions, and
$S_{jk}$ is the square root of the $d\times d$ matrix
$a_{jk}$, given by~\eqref{cov-B}. 
The limiting dynamics (\ref{020307}) has the following property proved
in Section~\ref{sec7.1}.
 {\begin{proposition}
\label{prop010507} Given $(t,x)$, and $s\in[t,T]$, let ${\mathfrak s}_{s}^{t,x}(u):=U^{t,x,u}(s)$.
The mapping ${\mathfrak s}_{s}^{t,x}
:\bbR\to\bbR$ is a.s. a diffeomorphism. 
\end{proposition}
This leads to the main result concerning the convergence of the solution of
\eqref{advec-11-semi2}.
\begin{theorem}
\label{thm011110}
(i) The joint laws of $\left(X_\eps^{t,x}(\cdot),{\mathfrak
    S}_{\eps}^{t,x}(\cdot)\right)$, over $C([t,T])\times
C([t,T]\times\bbR)$, equipped with the standard Frechet metric
metrizing uniform convergence on compact sets, converge weakly
to the law of $\left(X^{t,x}(\cdot),{\mathfrak
    S}^{t,x}(\cdot)\right)$, with ${\mathfrak
    S}^{t,x}(s,u):=U^{t,x,u}(s)$.\\ 
%
(ii) For each $(t,x)\in\bbR^{1+d}$ fixed, the random variables $u_\eps(t,x)$ converge in law,
as $\eps\to 0$, to  
\begin{equation}
\label{011001}
\mathscr{U}(t,x):=({\mathfrak
    s}_{T}^{t,x})^{-1}(u_0(X^{t,x}(T))).
\end{equation}
In addition, for
any positive integer $N$, mutually distinct $x_1,\ldots,x_N\in \bbR^d$
and $t\le T$, the random vectors
$(u_\eps(t,x_1),\ldots,u_\eps(t,x_N))$ converge in law, as $\eps\to 0$, to 
$ (\mathscr{U}^{(1)}(t,x_1),\ldots, \mathscr{U}^{(N)}(t,x_N))$,
where~$\mathscr{U}^{(1)},\ldots, \mathscr {U}^{(N)}$ correspond, via \eqref{011001}, to 
$({\frak s}^{t,x_j}_T(\cdot),X^{t,x_j}_j)$, $j=1,\ldots,N$ driven by i.i.d. copies of $d$-dimensional standard 
Brownian
motions as in \eqref{020307}. 
\end{theorem}}

\section{Some preliminaries on Gaussian, Markovian fields}
\label{s.markov}

In this section, we give a Markovian representation for the field $V(t,x)$, starting from the assumptions in Section~\ref{sec2.1}. 
To this end, let ${\cal H}_1$ be the $L^2$-closure of the linear space
spanned by the random variables 
\begin{equation}
\label{w-phi}
W(\varphi;w)=\sum_{j=1}^d \int_{\bbR^d}w_j(x)\varphi_j(x)dx,\quad \varphi\in\Sd,\quad w\in {\cal E},
\end{equation}
defined over the probability space $\left({\cal E},
  {\cal B}({\cal E}),\pi\right)$. Here, $\Sd$ is the space 
of divergence free vector fields $\varphi=(\varphi_1,\ldots,\varphi_d)$
with components in  ${\cal S}(\bbR^d)$.
By an approximation
argument,~$W$ extends to a unitary mapping
$W:H\to {\cal H}_1$, where 
$H$ is the (real) Hilbert space,  the closure of $\Sd$ in
the norm $\|\cdot\|_{H}$, with 
\begin{equation}
\label{l2s}
\langle\varphi_1,\varphi_2\rangle_{H}:=\int_{\bbR^d}\hat\varphi_1(k)\cdot\hat\varphi_2^*(k)\si(k)dk,\quad
\varphi_1,\varphi_2 \in \Sd.
\end{equation}
Here, $\sigma(k)$ is as in (\ref{vel}).
In addition, by \eqref{vel} and the fact that $\varphi_1,\varphi_2$ are divergence free, we have
\begin{equation}
\label{unitary}
\langle
W(\varphi_1;w),W(\varphi_2;w)\rangle_{L^2(\pi)}=\langle\varphi_1,\varphi_2\rangle_H,\quad
\varphi_1,\varphi_2\in H.
\end{equation}
Note that the shift 
\[
\tau_x\varphi(\cdot):=\varphi(x+\cdot)
\] is an isometry on  $H$,  for each $x\in\bbR^d$. In the following, we will simply write
\[
W(\varphi)=W(\varphi;w).
\]

\subsection{The Gaussian chaos expansion}

Let ${\cal P}_n$ be the space of the $n$-th
degree polynomials, the $L^2$-closure of the linear 
span of
$$
\prod_{j=1}^mW(\varphi_j),\quad 1\le m\le n ,\quad
\varphi_1,\ldots,\varphi_n\in H,
$$
and ${\cal H}_n:={\cal P}_n\ominus{\cal P}_{n-1}$, $n\ge1$ be the space of the $n$-th degree Hermite polynomials,
with the convention~${\cal H}_0={\cal P}_0=\bbR$.  It is well known, see
e.g. Theorem 2.6, p. 18 of \cite{janson}, that
\[
L^2(\pi)=\bigoplus_{n=0}^{+\infty}{\cal H}_n.
\]
Denote by ${\mathfrak p}_n$ the
orthogonal 
projection of $L^2(\pi)$ onto ${\cal H}_n$. Given $s\in[0,+\infty)$, the Hilbert space~${\mathfrak H}_s$
is made of   $F\in L^2(\pi)$ with ${\mathfrak p}_0F=0$ and the norm
\begin{equation}
\label{Hs}
\|F\|_{{\mathfrak H}_s}:=\Big\{\sum_{n=1}^{+\infty}(1+n)^{s}\|{\mathfrak p}_n F\|_{L^2(\pi)}^2\Big\}^{1/2}<+\infty.
\end{equation}
We set
\begin{equation}
\label{Hinf}
{\mathfrak H}_\infty:=\bigcap_{s\ge0}{\mathfrak H}_s.
\end{equation}

The homogeneity assumption on $\pi$ amounts to the fact that
$\pi\tau_x=\pi$ for each $x\in\bbR^d$. 
Therefore, the operators~$T_x F(w):=F(\tau_x w)$, $w\in{\cal E}$, $x\in\bbR^d$,
form  a strongly continuous group of isometries
 on~$L^p(\pi)$ for any $p\in[1,+\infty)$. Denote by  $D=(D_1,\ldots,D_d)$ the
generators of
$(T_x)_{x\in\bbR^d}$.
 Let ${\cal W}_{k,p}$  be the Banach space consisting of $F\in L^p(\pi)$ that 
belong to the domain of $D^m=\prod_{j=1}^dD_j^{m_j}$, for a
non-negative integer
multi-index $m=(m_1,\dots,m_d)$, with $|m|:=\sum_{i=1}^dm_i\le
k$, equipped with the norm
\begin{equation}
\label{wkp}
\|F\|_{k,p}:=\Big\{\sum_{|m|\le k}\|D^mF\|_{L^p(\pi)}^p\Big\}^{1/p}.
\end{equation}
The space ${\cal W}_{k,\infty}$ is defined with the help of $L^\infty$
norm.

We let
$
{\cal W}_\infty:=\bigcap_{p>1}{\cal W}_{1,p}.$
It follows from the definition of $T_x$   that
$$
T_x\Big(\prod_{j=1}^nW(\varphi_j)\Big)=\prod_{j=1}^nW(\tau_{-x}\varphi_j),\quad
\varphi_1,\ldots,\varphi_n\in H.
$$
Therefore,
$T_x({\cal P}_n)={\cal P}_n$, and since $T_x$ is unitary on
$L^2(\pi)$, we also get $T_x({\cal H}_n)={\cal H}_n$ for all $n\ge0$.
Due to the assumption that $\si$ is compactly supported, we
conclude easily that  ${\cal P}:=\bigcup_{n\ge0}{\cal P}_n\subset
{\cal W}_\infty$.

Finally,   we define the linear functionals
$v_p:{\cal E}\to\bbR$ as
\begin{equation}
\label{v}
v_p(w):=w_p(0),\quad
w\in {\cal E},\,p=1,\ldots,d.
\end{equation}
They are bounded  and can be
written as
\begin{equation}
\label{vp}
v_p=W(f_p),\quad p=1,\ldots,d,
\end{equation}
with $f_p\in H$ given by
$$
f_p(x):=
\int_{\bbR^d}e^{i k\cdot x}\Gamma(k) {\rm e}_pdk,~~
\hbox{${\rm
   e}_p=\mathop{\underbrace{(0,\ldots,1,\ldots,0)}}\limits_{\mbox{$p$-th
     position}}$.}
$$


\subsection{Markovian dynamics of the velocity field}

Here, we formulate the Markov property of the
${\cal E}$-valued
process $V_t:=V(t,\cdot)$, $t\in\bbR$.  We represent
the random field $V(t,x)$ in the form
\begin{equation}
\label{V}
V(t,x)=v(\tau_xV_t),\quad (t,x)\in\bbR^{1+d},
\end{equation}
with   $v=(v_1,\ldots,v_d)$ as in \eqref{v}.
Given $t\ge0$, let $S_t:H\to H$  be the continuous extension of 
\begin{equation}
\label{012801}
\widehat{S_t\varphi}(k):=e^{-\al(k)t}\widehat\varphi(k),\quad \varphi\in
\Sd.
\end{equation}
The family $(S_t)_{t\ge0}$ is 
a $C_0$-semigroup of symmetric contractions on $H$, with the
  generator $(-A)$ and 
\begin{equation}\label{nov1202bis}
\widehat{A\varphi}(k):=\al(k)\hat\varphi(k),  \quad \varphi\in \Sd.
\end{equation}
One can easily verify that $P_t$, defined via
\begin{equation}
\label{022801}
P_tW(\varphi):=W(S_t\varphi),\quad t\ge0,\quad \varphi \in H,
\end{equation}
forms a semigroup of 
contractions on ${\cal H}_1$.
Similarly, for
$F={\mathfrak p}_n\left(\prod_{j=1}^n W(\varphi_j)\right)$, we set
\begin{equation}
\label{hn}
P_tF:={\mathfrak p}_n\left(\prod_{j=1}^n W(S_t\varphi_j)\right).
\end{equation}
According to Theorem 4.5 of \cite{janson}, $\left(P_t\right)_{t\ge0}$ is a
contraction semigroup on ${\cal H}_n$ for each $n$, hence
a semigroup of contractions on the entire $L^2(\pi)$. 
It can be easily checked  (using e.g.
Theorem~3.9,~p.~26 of \cite{janson}) that $P_t$ forms a strongly
continuous semigroup of symmetric operators on $L^2(\pi)$.

Let ${\mathfrak V}_s$ be the $L^2$ closure of the linear span
of $W(\varphi;V_u)$ for any $u\le s$ and $\varphi\in H$. 
For any $t\ge s$ and
$\varphi\in H$, the orthogonal projection of
$W(\varphi;V_t)$ onto ${\mathfrak V}_s$ is 
$$
W(S_{t-s}\varphi;V_s)=P_{t-s}W(\varphi)(V_s).
$$
Therefore, according to Theorem 4.9, p. 46 of \cite{janson},
for any $F\in L^2(\pi)$  and $t\ge s$ we have
$$
\bbE\left[F(V_t)\left|\right.{\cal V}_s\right]=P_{t-s}F(V_s),
$$
where $\left({\cal V}_s\right)$ is the natural filtration of $\left(V_t\right)_{t\ge0}$.
Note that for any $x\in \bbR^d$, $t\ge0$ we have
\[
S_t(\tau_x(\varphi))=\tau_x(S_t\varphi).
\]
Using \eqref{hn},
we can conclude also that
\begin{equation}
\label{PT}
P_tT_x=T_xP_t,\quad t\ge0,\,x\in\bbR^d.
\end{equation}

It follows from \eqref{012801} and \eqref{022801} that
\begin{equation}
\label{022801a}
\|P_tF\|_{L^2(\pi)}\le e^{-\al_*t}\|F\|_{L^2(\pi)},\quad t\ge0,\quad F
\in {\cal H}_1.
\end{equation}
Using Theorem 4.5, p. 46 of \cite{janson} we conclude that
\eqref{022801a} actually holds for any $F\in L^2(\pi)$ such that~$\langle
F,1\rangle_{L^2(\pi)}=0$. We denote by $L:{\cal D}(L)\to L^2(\pi)$ the 
$L^2$-generator of $P_t$, which, due to the symmetry of the semigroup, is self-adoint. Since ${\cal P}$ is  dense in $L^2(\pi)$ and
invariant under $(P_t)_{t\ge0}$, it is a core of~$L$, see
e.g. Proposition 3.3, p. 17 of \cite{ethier-kurtz}.  As a consequence
of \eqref{022801a}
we have an estimate for the Dirichlet form
\begin{equation}
\label{022801b}
{\cal E}_L(F):=-\langle  LF,F\rangle_{L^2(\pi)}=-\lim_{t\to0} \frac{1}{t} \left(\langle P_t F,F\rangle -\langle F,F\rangle\right)\ge \al_*\|F\|_{L^2(\pi)}^2,\quad F
\in {\cal D}(L),\,F\perp 1.
\end{equation}
In fact, we have an estimate that allows us to compare
the Dirichlet form with the $L^2$ and $\|\cdot\|_{1,2}$ norms on the
space of the $n$-th degree Hermite polynomials.

\begin{theorem}
\label{thm013101}
The following estimates hold:
\begin{itemize}
\item[(i)]
\begin{equation}
\label{043101}
\al_*n\|F\|_{L^2(\pi)}^2 \le  {\cal E}_L(F)\le A_* n\|F\|_{L^2(\pi)}^2,\quad F\in {\cal
  H}_n,\quad n\ge0.
\end{equation}
\item[(ii)]
There exists a constant $C>0$ such that
\begin{equation}
\label{033101}
\sum_{j=1}^d\|D_jF\|_{L^2(\pi)}^2\le Cn {\cal E}_L(F),\quad F\in {\cal
  H}_n,\quad n\ge0.
\end{equation}
\item[(iii)]
There exists a constant $C>0$ such that
\begin{equation}
\label{042101}
\|v_p F\|_{L^2(\pi)}\le C{\cal E}_{L}^{1/2}(F),\quad
F\in {\cal H}_n,\quad n\ge 1,\quad p=1,\ldots,d.
\end{equation}
\end{itemize}
\end{theorem} 
The proof of part (i)  is presented in Section \ref{sec4.1}. The proofs
of parts (ii) and (iii) can be found in~\cite{KLO}, 
see the estimate~(12.115), p. 413  and Lemma 12.25,
p. 405, respectively.
As a direct conclusion from the above result, we obtain the following
(cf \eqref{Hs}).
\begin{corollary}
\label{cor011204}
We have ${\cal D}({\cal E}_L)={\mathfrak H}_{1}$. 
\end{corollary}

\subsection{A stochastic convolution representation for the velocity field} 
\label{sec2.6}

In order to obtain a more explicit representation for $V_t$, note that
given $\varphi\in H$, 
the process~$V_t(\varphi):=W(\varphi;V_t)$ 
is a Gaussian 
 semimartingale satisfying
\begin{equation}
\label{sde1}
dV_t(\varphi)=-V_t(A\varphi)dt+\sqrt{2}dB_t(\varphi),\quad t\ge
s,\,\varphi\in H,
\end{equation}
for any $s\in\bbR$.
Here, the  process $B:\bbR\times H
\times\Om\to\bbR$ is
such that the process
$
\left((B_t(\varphi_1),\ldots,B_t(\varphi_n)\right)_{t\in\bbR}$
is an $n$-dimensional, two sided, Brownian motion, with zero mean 
and covariance
\begin{equation}
\label{cB}
\E[B_t(\varphi_i)B_s(\varphi_j)]=(t\wedge s)\langle A \varphi_i,\varphi_j\rangle_H,\quad i,j=1,\ldots,n,\ \ t,s\in\bbR,
\end{equation}
for any $\varphi_1,\ldots,\varphi_n\in H$.
In addition, for any $s\in\bbR$ the process
$\left(B_t-B_s\right)_{t\ge s}$ is independent of ${\cal V}_s$ - the
$\si$-algebra generated by $V_u$, $u\le s$.

Suppose that an $H$-valued process $\left(\varphi_t\right)_{t\ge s}$
is progressively measurable w.r.t. the filtration ${\cal V}_t$ and satisfies
$$
\int_s^{t}\bbE\|\varphi_u\|_H^2 du<+\infty.
$$
By the standard procedure, we can define the It\^o integral
$$
\int_s^tdB_u(\varphi_u),
$$
sometimes also denoted by $$\int_s^t\langle \varphi_u,dB_u\rangle_H.$$
It is a square integrable, zero mean, continuous trajectory martingale that satisfies
$$
\bbE \left[\int_s^tdB_u(\varphi_u)\right]^2=\int_s^{t}\bbE\langle A\varphi_u,\varphi_u\rangle_H du.
$$
{Stationary, Gaussian and Markovian process $\left(V_t\right)$ can be
represented as a stochastic convolution.
\begin{proposition}
For any $\varphi\in H$ we can write
\begin{equation}
\label{011606}
V_t(\varphi)=\sqrt{2}\int_{-\infty}^t\langle {S_{t-s}}\varphi,dB_s\rangle_H,\quad
t\in\bbR.
\end{equation}
\end{proposition}}

\subsection{Proof of (\ref{043101})} \label{sec4.1}

Recall 
that the $n$-th degree, $L^2$-normalized, 
Hermite polynomial $h_n(x)$ is
$$
h_0(x)\equiv 1,\qquad
h_n(x):=\frac{(-1)^n}{\sqrt{n!}}e^{x^2/2}\frac{d^n}{dx^n}\left(e^{-x^2/2}\right),\quad x\in\bbR.
$$
It is well known that
\begin{equation}
\label{hermite}
xh_n(x)=(n+1)^{1/2}h_{n+1}(x)+n^{1/2}h_{n-1}(x),\quad
h_n'(x)=\sqrt{n}h_{n-1}(x),\quad n\ge1.
\end{equation}
Suppose that $\left({\mathfrak e}_j\right)_{j\ge1}$ is an orthonormal base
in $H$.  Let ${\bf n}=(n_j)_{j\ge1}$ be a sequence
of non-negative integers, 
and  $|{\bf n}|:=\sum_{j=1}^{+\infty}n_j$. 
According to Proposition 1.1.1 of \cite{nualart}, the vectors 
$$
h_{\bf n}:=\prod_{j=1}^{+\infty}h_{n_j}\left(W({\mathfrak
    e}_j)\right),\quad |{\bf n}|=n,
$$
form an orthonormal base in ${\cal H}_n$.
Suppose that $F=\sum_{n=0}^{+\infty}{\mathfrak p}_nF\in L^2(\pi)$ and 
$
{\mathfrak p}_n(F)=\sum_{|{\bf n}|=n}\al_{{\bf n}}h_{{\bf n}}
$
 for some real coefficients $\left(\al_{\bf n}\right)$ satisfying
$$
\sum_{n=0}^{+\infty}\sum_{|{\bf n}|=n}\al_{\bf n}^2=\|F\|_{L^2(\pi)}^2<+\infty.
$$

 {The operator ${\cal D}:{\mathfrak H}_1\to L^2(\pi;H)$ defined as 
\begin{equation}
\label{D}
{\cal D} F:=\sum_{j=1}^{+\infty} {\cal D}_j F{\mathfrak e}_j
\end{equation}
with ${\cal D}_j:{\mathfrak H}_1\to L^2(\pi)$, $j=1,2,\ldots$  given by
\begin{equation}
\label{Dj}
{\cal D}_j F:=\sum_{n=0}^{+\infty}\sum_{|{\bf n}|=n}\sqrt{n_j}\al_{{\bf n}}h_{{\bf n}_j},\quad F\in {\mathfrak H}_1,
\end{equation}
is the Malliavin derivative, see Definition 1.2.1, p. 25 of
\cite{nualart}. Here, ${\bf n}_j=(n_{j'}')$ is given by 
$$
n_{j'}':=
\left\{
\begin{aligned}
&n_{j'},\quad j'\not=j,\\
& (n_{j'}-1)_+,\quad j'=j.
\end{aligned}
\right.
$$
Note that, (cf Proposition 1.2.2, p. 28 of \cite{nualart})
\begin{equation}
\label{021610}
\sum_{j=1}^{+\infty}\|{\cal D}_j F\|_{L^2(\pi)}^2=\sum_{n=0}^{+\infty}n\sum_{|{\bf n}|=n}\al^2_{\bf n}=\sum_nn\|{\mathfrak p}_n F\|_{L^2(\pi)}^2<+\infty, \quad  F\in {\mathfrak H}_1.
\end{equation}}

 \noindent
 {{\bf Remark.}  For $F$ of the form
$
F:=\Phi(W(h_1),\ldots,W(h_N)),
$ 
where $h_1,\ldots,h_N\in H$ and~$\Phi\in C^\infty(\bbR^N)$ with
both $\Phi$ and its partial derivatives of polynomial
growth, we have
$
{\cal D} F=\sum_{p=1}^{N} \partial_{x_p} \Phi h_p.
$}

Denote by
$
h_{\bf n}(t):=h_{\bf n}(V_t)
$, and~$B_j(t):=B_t({\mathfrak e}_j)$, where $B_t$ was 
defined in Section \ref{sec2.6}.
Recall that (see \eqref{cB})
 $$
 \E[B_i(t)B_j(s)]=c_{i,j}(t\wedge s),\quad i,j=1,2,\ldots,
 $$
with  $c_{i,j}$ given by  
\begin{equation}\label{aij} 
c_{i,j}:=\langle A{\mathfrak e}_i,{\mathfrak e}_{j}\rangle_{H},\quad i,j=1,2,\ldots.
\end{equation}
Using the It\^o formula, \eqref{sde1} and \eqref{hermite}, one can show by a direct calculation that
\begin{align}
\label{010202}
dh_{\bf n}(t)=\Big\{&-\sum_{j}\sqrt{n_j}h_{{\bf n}_j}(t)V_t(A{\mathfrak
    e}_j)+\sum_{j_1\not=j_2} c_{j_1,j_2}\sqrt{n_{j_1}n_{j_2}}h_{{\bf
      n}_{j_1,j_2}}(t)\nonumber\\
&
+\sum_{j}c_{j,j}\sqrt{n_{j}(n_{j}-1)_+}h_{{\bf n}_{j}'}(t)\Big\}dt+\sqrt{2}\sum_{j}\sqrt{n_j}h_{{\bf n}_j}(t)dB_j(t).
\end{align}
Here, ${\bf n}_{j_1,j_2}=(m_{j'}')$, ${\bf n}_{j}'=(\ell_{j'}')$ are
multi-indices given by 
$$
m_{j'}':=
\left\{
\begin{aligned}
&n_{j'},\quad j'\not\in\{j_1,j_2\},\\
& (n_{j'}-1)_+,\quad j'\in\{j_1,j_2\},
\end{aligned}
\right.
\ell_{j'}':=
\left\{
\begin{aligned}
&n_{j'},\quad j'\not=j,\\
& (n_{j'}-2)_+,\quad j'=j. 
\end{aligned}
\right.
$$
%
 {For $F=\sum_n\sum_{|{\bf n}|=n}\al_{{\bf n}}h_{{\bf n}}$, by \eqref{hermite},  \eqref{aij} and \eqref{010202}, we have 
\begin{equation}\label{12221}
\begin{aligned}
\mathcal{E}_L(F)
=&\sum_{n,j}\sum_{|{\bf n}|=|{\bf m}|=n}\sqrt{n_j}\al_{{\bf
  n}}\al_{\bf m}\langle W(A{\mathfrak e}_j) h_{{\bf n}_j},h_{\bf m}\rangle_{L^2(\pi)} \\
 =&\sum_{n,j,j'} \sum_{|{\bf n}|=|{\bf m}|=n}c_{j,j'}\sqrt{n_j}\al_{{\bf
  n}}\al_{\bf m}\langle h_{{\bf n}_j},  W({\mathfrak e}_{j'}) h_{\bf m}\rangle_{L^2(\pi)}\\
  =&\sum_{n,j,j'}\sum_{|{\bf n}|=|{\bf m}|=n}c_{j,j'}\sqrt{n_j  m_{j'}}\al_{{\bf
  n}}\al_{\bf m}\langle h_{ {\bf
  n}_{j}},h_{{\bf m}_{j'}}\rangle_{L^2(\pi)}.
\end{aligned}
\end{equation}
Comparing with \eqref{D} and \eqref{Dj} we conclude the formula
\begin{equation}
\label{021610bis}
{\cal E}_L(F)=\int_{\cal E}\langle A{\cal D} F,{\cal D} F\rangle_{H}d\pi, \quad
F\in D({\cal E}_L).
\end{equation}}
%
 {Thanks to the inequality
\begin{align}
\label{020202}
\al_*\|\varphi\|_H^2\le 
\langle A\varphi,\varphi\rangle_H\le
A_*\|\varphi\|_H^2,\quad \forall\,\varphi\in H,
\end{align}
(following directly from \eqref{nov1202}) and identity \eqref{021610} we conclude that 
\begin{equation}
\label{010601-18}
\al_*\sum_nn\|{\mathfrak p}_n F\|_{L^2(\pi)}^2\le {\cal
  E}_L(F)\le A_*\sum_nn\|{\mathfrak p}_n F\|_{L^2(\pi)}^2,\quad F\in {\cal
  D}({\cal E}_L).
\end{equation}
 Hence $F$ belongs to ${\cal
  D}({\cal E}_L)$ -- the domain of the form ${\cal E}_L(\cdot)$ iff
$F\in {\frak H}_1$, i.e.
$$
\sum_nn\|{\mathfrak p}_n F\|_{L^2(\pi)}^2<+\infty.
$$
Thus, in particular \eqref{043101} follows.\qed}

\subsection{Some corollaries of Theorem \ref{thm013101}}

Note that $F\in {\cal D}(L)$ iff $G:=L^{1/2}F\in {\cal D}({\cal
  E}_L)$, that is $G\in {\mathfrak H}_1$. However, according
to  part~(i)~of Theorem \ref{thm013101} then $\|{\mathfrak
  p}_n G\|_{L^2(\pi)}^2\asymp n\|{\mathfrak p}_nF\|_{L^2(\pi)}^2$.  {The
symbol $a_n\asymp b_n$ used for two non-negative sequences
$(a_n)_{n\ge1}$ and $(b_n)_{n\ge1}$ means that there exists $C>0$ such
that $C a_n\le b_n \le  a_n/C$ for all $n$.} We conclude the following.
\begin{corollary}
\label{L} 
We have
${\cal D}(L)={{\mathfrak H}}_2$.
\end{corollary}

Next, we write down an It\^o formula
for the process $V_t$ that will also be of great use for us.
 {From~\eqref{010202} we obtain that for any $F\in {\cal D}(L)$
\begin{equation}
\label{f}
F(V_t)=F(V_0)+\int_0^tLF(V_s)ds+\sqrt{2}M_t(F),
\end{equation}
where
$M_t(F)$ is a continuous, square integrable
martingale given by
\begin{equation}
\label{mf}
M_t(F)=\sum_{j=1}^{+\infty}\int_0^t{\cal D}_j
F(V_s)dB_j(s),
\end{equation}
where ${\cal D}F$ is the Malliavin derivative defined in \eqref{D}.}

\section{The environment process and the corrector fields}
\label{s.corrector}

Let $X^{t,x}(s)$ be the solution of \eqref{ode} corresponding to $\eps=1$.
The (${\cal E}$-valued)
 {\em
  environment process} is
\begin{equation}
\label{eta}
\eta_s^{t,x}:=\tau_{X^{t,x}(s)}V_s,  \quad s\geq t,
\end{equation}
so that
$$
X^{t,x}(s)=x+\int_t^sv(\eta^{t,x}_\si)d\si.
$$
We shall write $X(s)$, $\eta_s$ instead of $X^{0,0}(s)$ and
$\eta_s^{0,0}$, respectively. 

\subsection{Properties of the environment process and the corrector}
 
\label{sec4.1a} 

Let ${\cal V}_{t,s}$ be the $\si$-algebra generated by $V_u$, $t\le
u\le s$ and $B_b({\cal E})$ be the space of 
bounded Borel
measurable  functions
$F:{\cal E}\to\bbR$. 
The following  is a consequence of  the results  
in~\cite[Section~12.10]{KLO}.
\begin{proposition}
\label{prop012908}
For a given $(t,x)$, 
the natural filtration of $(\eta_s^{t,x})_{s\ge t}$  coincides with 
$\left({\cal V}_{t,s}\right)_{s\ge t}$.
The process  $(\eta_s^{t,x})_{s\ge t}$ is Markovian and
stationary, that is, for any  $s\ge t$ and $h\ge0$ we have
$$
\bbE\left[F(\eta_{s+h}^{t,x})\left|\right.{\cal
    V}_{t,s}\right]=Q_hF(\eta_{s}^{t,x}),\quad\mbox{ a.s.},\, \mbox{where }F\in B_b({\cal E}),
$$
 and 
$$
Q_sF(w)=\bbE\left[F(\eta_s)\left|\right.\eta_0=w\right],\quad \,s\ge0,\,\pi \mbox{ a.s. in }w\in {\cal E}, \, \mbox{where } F\in
L^2(\pi).
$$
In addition, $\pi$ is invariant under $(Q_s)_{s\ge0}$:
\begin{equation}
\label{011610}
\int_{\cal E}Q_s Fd\pi=\int_{\cal E} Fd\pi,\quad s\ge0 ,\,F\in B_b({\cal E}),
\end{equation}
and $(Q_s)_{s\ge0}$ extends to
 $C_0$-continuous semigroup of contractions on $L^2(\pi)$. 
\end{proposition}
Denote by ${\cal L}$ the generator of the semigroup $Q_s$ on $L^2(\pi)$. Recall that ${\cal P}$ is the set of all polynomials. The following is proved in Section \ref{sec4.2}. 
\begin{proposition}
\label{core}
The set ${\cal P}$
 is a common core of both  ${\cal L}$ and  $L$. In addition, we have
\begin{equation}
\label{cal-L}
{\cal L}F=LF+\sum_{j=1}^dv_jD_j F,\quad F\in {\cal P}.
\end{equation}
\end{proposition}
We know from \eqref{cal-L}, \eqref{043101} and the fact that $v$ is divergence free (cf. \cite[Corollary~12.22]{KLO})
\begin{equation}
\label{cal-L1}
-\langle {\cal L}F,F\rangle_{L^2(\pi)}=-\langle
LF,F\rangle_{L^2(\pi)}\ge \al_*\|F\|_{L^2(\pi)}^2,\quad F\in {\mathfrak
  H}_\infty,\quad F\perp 1.
\end{equation}
This implies the exponential stability of the semigroup in $L^2(\pi)$: 
\begin{equation}
\label{sg1}
\Big\|Q_t F-\int_{\cal E}Fd\pi\Big\|_{L^2(\pi)}\le e^{-\al_*
  t}\Big\| F-\int_{\cal E}Fd\pi\Big\|_{L^2(\pi)},\quad F\in L^2(\pi),\,t\ge0.
\end{equation}
Combining \eqref{011610} with \eqref{sg1} we conclude, via an
interpolation  between $L^2(\pi)$ and $L^1(\pi)$, that
\begin{equation}
\label{sg1a}
\Big\|Q_t F-\int_{\cal E}Fd\pi\Big\|_{L^p(\pi)}\le 
e^{-2\al_*(1-1/p)t}\Big\| F-\int_{\cal E}Fd\pi\Big\|_{L^2(\pi)},
\quad F\in L^p(\pi),\,t\ge0,\,p\in[1,2],
\end{equation}
and by interpolation between $L^2(\pi)$ and $L^\infty(\pi)$  
also that
\begin{equation}
\label{sg1b}
\Big\|Q_t F-\int_{\cal E}Fd\pi\Big\|_{L^p(\pi)}\le e^{-2\al_*t/p}
\Big\| F-\int_{\cal E}Fd\pi\Big\|_{L^2(\pi)},\quad 
F\in L^p(\pi),\,t\ge0,\,p\in[2,+\infty).
\end{equation}
 {\begin{theorem}
\label{prop013101}
Suppose that 
$F_*\in L^p(\pi)$ for some $p\in(1,+\infty)$ and $\int_{\cal
  E}Fd\pi=0$. Then,
the equation
\begin{equation}
\label{corr1}
- {\cal L}\theta=F_*
\end{equation}
admits a unique zero mean solution $\theta$ that belongs to the domain of
the $L^p$-generator ${\cal L}$. In addition, if $F_*\in {\frak H}_s$ for some
$s>0$, then $\theta\in  {\frak H}_{s+2}$.
\end{theorem}
The proof of this theorem is  in Section \ref{sec4.2}.
Since each $v_j\in {\frak H}_{\infty}$, as an immediate consequence of
Theorem \ref{prop013101}, we conclude the following.
\begin{corollary}
\label{cor11a}
The equation
\begin{equation}
\label{cor1}
- {\cal L}\chi_j=v_j,\quad j=1,\ldots,d,
\end{equation}
admits a unique solution $\chi_j\in {\cal D}({\cal L})\cap {\mathfrak
  H}_\infty$ and $\chi_j\perp1$ for each $j=1,\ldots,d$.  
\end{corollary}}
The solutions of  \eqref{cor1} are known as the correctors. They can be used to express the 
effective diffusivity matrix appearing  in  the homogenized
equation \eqref{bheat-11}: 
\begin{equation}
\label{apq}
a_{ij}:=\langle v_i,\chi_{j}\rangle_{L^2(\pi)}={\cal
  E}_L(\chi_i,\chi_{j}),\quad i,j,=1,\ldots,d,
\end{equation}
with
\begin{equation}
\label{021610bis2}
{\cal E}_L(F,G):=-\langle LF,G\rangle_{L^2(\pi)}=
\int_{\cal E}\langle A{\cal D} F,{\cal D} G\rangle_{H}d\pi.
\end{equation}
We define the corrector fields as 
stationary in $(t,x)$ random fields
$\tilde\chi_j:\bbR^{1+d}\times\Om\to\bbR$, given by
\begin{equation}
\label{cor-fields1}
\tilde\chi_j(t,x;w):=\chi_j(\tau_xw_t),\quad (t,x)\in\bbR^{1+d},\,j=1,\ldots,d.
\end{equation}
Combining the results 
of Theorems \ref{prop013101} and \ref{thm013101}, we
conclude the following
\begin{corollary}
\label{cor013101}
The fields $\tilde\chi_j$, $\nabla_x\tilde\chi_j$ are square
integrable for each $j=1,\ldots,d$:
\begin{equation}
\label{053101}
\sum_{i=1}^d\Big\{\bbE\tilde\chi_i^2(0,0)+\bbE\Big[|\nabla_x\tilde\chi_i(0,0)|^2\Big]\Big\}<+\infty.
\end{equation}
\end{corollary}




\subsection{The ``far away'' independence }

In order to deal with the spatial decorrelation properties of the velocity field,
 note that for each $x\in\bbR^d$
fixed,  the set $\{{\mathfrak e}_n^x:=\tau_{-x}{\mathfrak e}_n\}$ is an
orthonormal base on $H$,
and we can write
$$
{\mathfrak e}_n^x=\sum_{m=1}^{+\infty}u_{nm}(x){\mathfrak e}_m.
$$
Here, $[u_{nm}(x)]$ is an infinite orthogonal matrix with
\begin{equation}
\label{010901}
u_{nm}(x)=\langle {\mathfrak e}_n^x,{\mathfrak
  e}_m\rangle_H=\int_{\bbR^d}e^{-ik\cdot x}\hat {\mathfrak
  e}_n(k)\cdot\hat{\mathfrak e}_{m}^*(k)\si(k)dk,\quad n,m\ge1.
\end{equation}
As $\si(k)$ is compactly supported, each $u_{nm}$ is bounded and analytic. We also have
\[
u_{nm}(0)=\delta_{m,n},~u_{nm}(-x)=u_{mn}(x),
\]
and 
\begin{equation}
\label{021105}
\sum_{k=1}^{+\infty}u_{nk}(x)u_{km}(y)=u_{nm}(x+y),\quad x,y\in\bbR^d,\,m,n\ge1.
\end{equation}
We will use the following ``decorrelation lemma''.
\begin{proposition}
\label{prop012705}
Suppose that $\al$, $\si$ satisfy the assumptions in Section \ref{sec1.1}, for any $n,m\geq1$, let 
\begin{equation}
\label{vj}
v_{nm}(x):=\langle A {\mathfrak e}_n^x,{\mathfrak e}_m\rangle_H,
\end{equation}
then
\begin{equation}
\label{012705}
\lim_{|x|\to+\infty}v_{nm}(x)=0.
\end{equation}
\end{proposition}
\proof
We have
$$
v_{nm}(x)=\langle A\tau_{-x}{\mathfrak e}_n,{\mathfrak e}_m\rangle_H=\int_{\bbR^d}e^{-ik\cdot x}\hat{\mathfrak e}_n(k)\hat{\mathfrak e}^*_m(k)\al(k)\si(k)dk.
$$
The result is an immediate consequence of  the Riemann-Lebesgue lemma.
\qed


\subsection{The It\^o formula for the environment process} 

To obtain the It\^o formula for $\eta_t$, suppose that $\varphi\in H$ and $(t,x)\in\bbR^{1+d}$ and let
$$
\hat B^{t,x}_s(\varphi):=
\int_t^s\langle\tau_{-X^{t,x}(\si)}\varphi, d B_\si\rangle_H.
$$
Define
\begin{align}
\label{hat-B1bis}
\hat B_j^{t,x}(s):=\hat B^{t,x}_s({\mathfrak e}_j)=
\sum_{k=1}^{+\infty}\int_t^{s}u_{jk}\left(X^{t,x}(\sigma)\right)dB_{k}(\si),\quad j=1,2,\ldots,
\end{align}
where, as we recall $ B_j(t):=B_t({\mathfrak e}_j)$.
The following result holds.
 {\begin{corollary}
\label{L-cl}
The space of polynomials ${\cal P}$  is a common core of ${\cal D}( L)$ and  ${\cal D}({\cal
  L})$. In addition, for any $F\in {\cal D}({\cal
  L})$ we have $F\in {\cal D}({\cal
  E}_L)$ and
\begin{equation}
\label{f1x}
\langle (-{\cal L})F,F\rangle_{L^2(\pi)}={\cal E}_L(F).
\end{equation}
In addition, for any 
$F\in {\cal D}({\cal
  L})$
 and~$(t,x)\in \bbR^{1+d}$  the following It\^o formula holds
\begin{equation}
\label{f1}
F(\eta_s^{t,x})=F(\tau_xV_t)+\int_t^s{\cal
  L}F(\eta_\si^{t,x})d\si+\sqrt{2}\hat M_s^{t,x}(F),
\end{equation}
where
$\left(M_s^{t,x}(F)\right)_{s\ge t}$ is a continuous square integrable
martingale given by
\begin{equation}
\label{mf1bis}
\hat M_s^{t,x}(F):=\sum_{j=1}^{+\infty}\int_t^s{\cal D}_j F(\eta_\si^{t,x}) d\hat
B_j^{t,x}(\si)=\int_{t}^{s} \left\langle \tau_{-X^{t,x}(\si)}{\cal D} F(\eta_\si^{t,x}),dB_{\si}\right\rangle_H,
\end{equation}
with ${\cal D}_j$, $j=1,2,\ldots$ and ${\cal D}$ given by \eqref{Dj}
and \eqref{D}, respectively.
\end{corollary}
\proof
The first part follows from
Proposition \ref{core}. 
Formula \eqref{f1x} holds for $F\in{\cal P}$, as can be easily seen by
an application of \eqref{cal-L}. The extension to $D({\cal L})$ can be
done by an approximation.}

 {For any $F\in {\cal P}$ the formula \eqref{mf1bis} follows from \eqref{mf} and the definition of the process
$\left(\eta_s^{t,x}\right)_{s\ge t}$, see \eqref{eta}. The extension
to an arbitrary $F\in D({\cal L})$ can, again,  be
achieved by an approximation argument.\qed}

\subsection{Proofs of Proposition \ref{core} and Theorem \ref{prop013101}}

\label{sec4.2}


\subsubsection*{Proof of Proposition \ref{core}}

Since ${\cal P}$ is dense in $L^2(\pi)$ and  invariant under the semigroup $P_t$,
it is a core of ${\cal D}(L)$.
By a direct
calculation using the It\^o formula \eqref{f},  it can be checked that
${\cal P}\subset {\cal D}({\cal L})$ and the action of ${\cal L}$ on~$F\in{\cal P}$ 
is given by \eqref{cal-L}.
In what follows, we verify that in fact ${\mathfrak H}_4 \subset {\cal D}({\cal L})$
and  \eqref{cal-L} holds also for any $F\in {\mathfrak H}_4$.
Then, \eqref{cal-L1} also holds for all $F\in{\mathfrak H}_4$, so in particular
${\cal L}$ is dissipative on~${\cal P}$, i.e.   { for any $\la>0$ we have $\|(\la-{\cal
  L})F\|_{L^2(\pi)}\ge {\la}\|F\|_{L^2(\pi)}$, $F\in {\cal P}$.}
Using Theorem 2.12, p. 16
of \cite{ethier-kurtz} we conclude that $\tilde{\cal
  L}$, the  closure of  ${\cal L}$, restricted to ${\cal P}$,
is a generator of a strongly continuous semigroup on $L^2(\pi)$. But
${\cal L}$ itself is closed (as a generator of a $C_0$-semigroup)
therefore $\tilde{\cal L}\subset {\cal L}$. The latter in turn implies that
$\tilde{\cal L}={\cal L}$, as then we have~$
(\la-\tilde{\cal L})^{-1}=(\la-{\cal L})^{-1}
$
for any $\la>0$. In particular, the above means that 
${\cal P}$ is a  core of ${\cal L}$, which ends the proof of
Proposition \ref{core}. It remains to show that~\eqref{cal-L} holds for
$F\in{\mathfrak H}_4$ and the density of  $(\la-{\cal L})({\cal P})$  in  $L^2(\pi)$.

Recall that  $F\in {\mathfrak H}_4$ iff
\begin{equation}
\label{021204}
\sum_{n=0}^{+\infty} (n+1)^4\|{\mathfrak p}_n F\|_{L^2(\pi)}^2<+\infty.
\end{equation}
Thanks to \eqref{043101}, we conclude that 
\begin{equation}
\label{013108}
\|L{\mathfrak p}_n
F\|_{L^2(\pi)}\leq A_*n \|{\mathfrak p}_n
F\|_{L^2(\pi)}, \quad n=1,2,\ldots
\end{equation}
Let
\begin{equation}
\label{F-n}
F_n:=\sum_{k=0}^n{\mathfrak p}_kF,
\end{equation}
 then
$F_n  \in {\cal D}({\cal L})$ and 
$$
{\cal L}F_n=L F_n+\sum_{j=1}^2 v_jD_j F_n.
$$ 
Using the fact that $F\in {\mathfrak
  H}_4$ and \eqref{013108}, we conclude that  $LF_n\to LF$, as
$n\to+\infty$. 
Next, we show that $v_jD_j F_n$ converges in $L^2(\pi)$ for each $j=1,\ldots,d$.
Thanks to the first formula in \eqref{hermite} we have
$$
{\mathfrak p}_m\left(v_jD_j {\mathfrak p}_k F_n\right)=0,\quad |k-m|\not=1.
$$
Hence, for $n'>n$, using orthogonality we have
\begin{align}
\label{031204}
\Big\|\sum_{m,k=0}^{+\infty}{\mathfrak p}_m\left(v_jD_j {\mathfrak p}_k (F_{n'}-F_n)\right)\Big\|_{L^2(\pi)}\le &\Big\{\sum_{m=0}^{+\infty}\left\|{\mathfrak p}_m\left(v_jD_j {\mathfrak p}_{m+1} (F_{n'}-F_n)\right)\right\|_{L^2(\pi)}^2\Big\}^{1/2}\nonumber\\
&
+\Big\{\sum_{m=1}^{+\infty}\left\|{\mathfrak p}_{m+1}\left(v_jD_j {\mathfrak p}_{m} (F_{n'}-F_n)\right)\right\|_{L^2(\pi)}^2\Big\}^{1/2}.
\end{align}
Using the H\"older inequality, we conclude that for $m>1$
$$
\left\|{\mathfrak p}_m\left(v_jD_j {\mathfrak p}_{m+1} (F_{n'}-F_n)\right)\right\|_{L^2(\pi)}\le \|v_j\|_{L^{2m}(\pi)}\left\|D_j {\mathfrak p}_{m+1} (F_{n'}-F_n)\right\|_{L^{2m/(m-1)}(\pi)}.
$$
Since $v_p$ is Gaussian, we have
$
\|v_p\|_{L^{2m}(\pi)}\sim (m!)^{1/(2m)},
$ which, by virtue of Stirling's formula, is of the order $\sqrt{m}$. 
On the other hand, by the hypercontractivity of $L^p$ norms with respect 
to a Gaussian  measure, see e.g. Theorem 5.10 of \cite{janson}, we have 
$$
\left\|D_j {\mathfrak p}_{m+1} (F_{n'}-F_n)\right\|_{L^{2m/(m-1)}(\pi)}\le \left(\frac{m+1}{m-1}\right)^{m/2}\left\|D_j {\mathfrak p}_{m+1} (F_{n'}-F_n)\right\|_{L^2(\pi)}.
$$
Therefore, there exists $C>0$ such that
\begin{align*}
&\left\|{\mathfrak p}_m\left(v_jD_j {\mathfrak p}_{m+1}
    (F_{n'}-F_n)\right)\right\|_{L^2(\pi)}\le C\sqrt{m}\left\|D_j
  {\mathfrak p}_{m+1} (F_{n'}-F_n)\right\|_{L^2(\pi)}\\
&
\le C\sqrt{m}(m+1)\left\| {\mathfrak p}_{m+1}
  (F_{n'}-F_n)\right\|_{L^2(\pi)},\quad m\ge1,
\end{align*}
by virtue of parts (i) and (ii) of Theorem \ref{thm013101}. A similar estimate 
holds for the second term in the right hand side of \eqref{031204}. 
As a result,  there exists $C>0$ such that
$$
\Big\|\sum_{j=1}^d v_jD_j (F_{n'}-F_n)\Big\|_{L^2(\pi)}\le 
C\|F_{n'}-F_n\|_{{\mathfrak H}_3},\quad n'>n.
$$
The above implies that ${\cal L}F_n$ converges in $L^2(\pi)$. Thus, the right
side of \eqref{cal-L} makes sense
for any~$F\in {\mathfrak H}_4$, so that $F\in {\cal
  D}({\cal L})$ and the action of ${\cal L}$ on ${\mathfrak H}_4$ is given by \eqref{cal-L}.

To show the density of $(\la-{\cal L})({\cal P})$ in $L^2(\pi)$ we
observe first that $(\la-{\cal L})({\mathfrak H}_\infty)$ is dense in $L^2(\pi)$.
Indeed,  Lemma~2.21, p.~63 of~\cite{KLO} implies that
 given any~$G\in {\cal P}$, there exists $F\in {\mathfrak H}_\infty$ such that it satisfies the resolvent
 equation $(\la- {\cal L})F=G$. Given $F\in {\mathfrak H}_\infty$ we let
 $F_n\in {\cal P}$ be defined by~\eqref{F-n}. The previous argument shows that
 $F_n\to F$ and ${\cal L}F_n\to{\cal L}F$, as $n\to+\infty$, in
 $L^2(\pi)$ (it even holds for $F\in {\mathfrak H}_4$).  This proves that
 the closure of $(\la-{\cal L})({\cal P})$ equals $L^2(\pi)$. The
 proof of Proposition \ref{core} is therefore complete.\qed


\subsubsection*{Proof of Theorem \ref{prop013101}}

 {The zero mean solution of  \eqref{corr1} is given by 
\[
\theta=\int_0^{+\infty}Q_tF_*dt,
\]
and the integral in the right side converges,   thanks to \eqref{sg1}.}

 {
In light of the already proved Proposition~\ref{core},  the generator ${\cal L}$ can be written  on its core ${\cal P}$ as~$L+{\cal A}$, 
where $L$ is the generator of~$V_t$, that is essentially self-adjoint on ${\cal P}$ (which is its core), and ${\cal A}F=\sum_{p=1}^d v_p D_p F$, 
$F\in {\cal P}$, is antisymmetric. We see from \eqref{033101} and \eqref{042101} that there exists a constant $C>0$ such that
$$
 |\langle {\cal A}F,G \rangle_{L^2(\pi)} |\le \sum_{p=1}^d\|D_pF\|_{L^2(\pi)}\|v_pG\|_{L^2(\pi)}\le C(n+1)^{1/2}{\cal E}_L^{1/2}(F){\cal E}_L^{1/2}(G)
$$
for any $F\in {\cal H}_n$ and $G\in {\cal H}_{n+1}$, or $G\in {\cal
  H}_n$ and $F\in {\cal H}_{n+1}$, and $n=0,1,\ldots$. }

 {If $F_*\in {\frak H}_s$ for some $s>0$ then
$$
\sum_{n=0}^{+\infty}(n+1)^{s+1}\|\!|{\frak p}_nF_*\|\!|_{-1}^2<+\infty,
$$
where
$$
\|\!|{\frak p}_nF_*\|\!|_{-1}^2:=\sup_{G}\left[2\langle {\frak
    p}_nF_*,G\rangle_{L^2(\pi)}-{\cal E}_L(G)\right]\asymp
\frac{\|{\frak p}_nF_*\|_{L^2(\pi)}^2}{n+1},\quad n=1,2,\ldots,
$$
by virtue of part (i) of Theorem \ref{thm013101}.
By virtue of
Lemma 2.21, p. 67 of \cite{KLO}, for any $s\ge 1$ the zero-mean solution of
\eqref{corr1} satisfies
$$
\sum_{n=0}^{+\infty}(n+1)^{s+1}{\cal E}_L({\frak p}_n\theta)<+\infty,
$$
which, by another application of  Theorem \ref{thm013101}, shows that
 $\theta\in {\mathfrak H}_{s+2}$ (as ${\cal E}_L({\frak p}_nF_*)\asymp
 n \|{\frak p}_nF_*\|_{L^2(\pi)}^2$), which ends the proof of the theorem.\qed}

\section{Proof of Theorem \ref{thm}}

\label{sec4.3}



 To avoid cumbersome notations we consider only the case $N=2$, as the
  general case can be argued using the same proof as  
 below. Theorem \ref{thm} is an immediate corollary of  the
following result.
\begin{theorem}
\label{thm011105}
For any $(t,x,y)\in\bbR\times \bbR^{2d}$ with $x\not=y$, the processes
$\left(X_\eps^{t,x}(s),X_\eps^{t,y}(s)\right)_{s\geq t}$ converge weakly over $C([t,+\infty);\bbR^{2d})$ to $\left(x+\beta_{s-t},y+\beta'_{s-t}\right)_{s\ge t}$, where
$(\beta_t)_{t\ge0}$, $(\beta_t')_{t\ge0}$ are two independent copies of Brownian motion with the covariance matrix as in \eqref{cov-B}.
\end{theorem}
This result is not very surprising -- two particles starting at two different positions will see ``nearly independent'' environments. 
However, as the realizations of the velocity field in our case are analytic in space, the argument is slightly more delicate than,
say, for velocity fields with finite range dependence, and relies on Proposition~\ref{prop012705} rather than the usual mixing properties.  
As we have mentioned,
convergence of each individual trajectory to a Brownian path is well known under our assumptions.
 
\subsubsection*{Decomposition of the trajectory}
   
Let 
\[
\eta_{\eps,s}^{t,x}:=\tau_{X^{t,x}_\eps(s)/\eps}V_{s/\eps^2},
\]
 then, using
\eqref{cor1} and \eqref{f1}, we can decompose the $k-$th component of $X^{t,x}_\eps(s)$, denoted by $X^{t,x}_{k,\eps} (s)$, as 
\begin{eqnarray}
\label{022704}
&&X^{t,x}_{k,\eps} (s)
=x_k+\frac{1}{\eps}\int_t^s v_k (\eta_{\eps,\sigma}^{t,x})d\sigma
=x_k-\frac{1}{\eps}\int_t^s{\cal L}\chi_k
\big(\eta_{\eps,\sigma}^{t,x}\big)d\sigma \\
&&~~~~~~~~~~~
=x_k+\eps Y_{k,\eps}(t,s)+\sqrt{2}\int_t^s\langle \tau_{-X^{t,x}_\eps(\si)/\eps}{\cal D}\chi_k(\eta_{\eps,\sigma}^{t,x}),dB^\eps_\si\rangle_H
=x_k+\eps Y_{k,\eps}(t,s)+\sqrt{2}M_{\eps,k}^x(t,s),\nonumber
\end{eqnarray}
where
$$
Y_{k,\eps}(t,s):=\chi_k
(\tau_{x/\eps} V_{t/\eps^2})-{\chi_k}
(\eta_{\eps,s}^{t,x}),~~M_{\eps,k}^x(t,s):=
 {\sum_{j=1}^{+\infty}\int_t^s
  {\cal D}_j\chi_k
 (\eta_{\eps,\sigma}^{t,x})
 d\hat  B_{j,\eps}^{t,x}(\sigma)}, 
$$
and  $\hat B_{j,\eps}$ are defined using the change of variables \eqref{010901} and \eqref{hat-B1bis} 
\begin{align}
\label{hat-B}
\hat B_{j,\eps}^{t,x}(s):=
\sum_{k=1}^{+\infty}\int_t^{s}u_{jk}\Big(\frac{X_\eps^{t,x}(\sigma)}{\eps}\Big)dB_{k}^{\eps}(\si),~
B_j^{\eps}(t):=\eps B_j(t/\eps^2),\quad j=1,2,\ldots
\end{align}
By Corollary~\ref{cor013101}, the main contribution to $X^{t,x}_{k,\eps} (s)$ in (\ref{022704})
comes from $M_{\eps,k}^x(t,s)$.  In fact, one can show the following.
 {\begin{proposition}
\label{prop010601}
For any $(t,x)\in [0,T]\times\bbR^d$ and $\delta>0$ we have
\begin{equation}
\label{020601}
\lim_{\eps\to0+}\bbP\left[\eps\sup_{t\le s\le T}|{\chi_k}
(\eta_{\eps,s}^{t,x})|\ge \delta\right]=0,\quad k=1,\ldots,d.
\end{equation}
\end{proposition}
\proof Due to stationarity, it suffices only to show that for each
$k=1,\ldots,d$
\begin{equation}
\label{020601a}
\lim_{\eps\to0+}\eps\sup_{0\le s\le T/\eps^2}|{\chi_k}
(\eta_{s})|=\lim_{\eps\to0+}\left\{\eps^2\max_{0\le k\le [T/\eps^2]}{\cal X}_k\right\}^{1/2}
=0,\quad \bbP\quad \mbox{a.s., }
\end{equation}
where
${\cal X}_k:=\sup_{k\le s\le k+1}|{\chi_k}(\eta_{s})|^2$. The sequence
is stationary and ergodic, thanks to the results of Section \ref{sec4.1a}. 
The Ito formula \eqref{f1} applied to $\chi_k$ implies that
\begin{equation}
\label{050601}
\bbE {\cal X}_1<+\infty.
\end{equation}
We claim that
\begin{equation}
\label{020601aa}
\lim_{N\to+\infty}\frac{1}{N}\max_{1\le k\le N}{\cal X}_k=0, \quad \bbP\quad \mbox{a.s., }
\end{equation}
which in turn yields \eqref{020601a}.}

 {
Indeed,
note first that ${\cal X}_N/N\to0$, $\bbP$ a.s. Indeed, by the stationarity
and ergodicity of the sequence $({\cal X}_N)_{N\ge1}$ and the  Birkhoff
individual ergodic theorem we have
$$
\frac{{\cal X}_N}{N} =\frac{1}{N}\sum_{k=1}^N{\cal X}_k-\frac{N-1}{N}\left(\frac{1}{N-1}\sum_{k=1}^{N-1}{\cal X}_k\right)\to0, \quad \bbP\quad \mbox{a.s. }
$$
Then
\begin{equation}
\label{020601aaa}
\lim_{N\to+\infty}\frac{1}{N}\max_{1\le k\le N}{\cal
  X}_k=\lim_{N\to+\infty}\max_{1\le k\le N}\left[\frac{k}{N}\cdot \frac{{\cal X}_k}{k}\right] =0, \quad \bbP\quad \mbox{a.s. }
\end{equation}}
\qed


\subsubsection*{Decorrelation properties for separated trajectories}

Next, we show that if trajectories are ``slightly separated'' then they have a small co-variation in a certain sense.
We assume without loss of generality that 
\[
y=0, \   \  t=0,
\] 
and set 
\[
M_{\eps}^x(s)=(M_{\eps,1}^x(s),\ldots,M_{\eps,d}^x(s)),~~M_{\eps,k}^x(s):=M_{\eps,k}^x(0,s).
\] 
Let ${\cal Q}_{x}^\eps$ be the joint law of   
$(X_\eps^{0,x}(s),X_\eps(s))_{s\ge 0}$ over
${\cal C}_{2d}:=C([0,+\infty);\bbR^{2d})$, where $X_\eps(s)=X_\eps^{0,0}(s)$. 
We know that each of the components $X_\eps^{0,x}(s) $ and $X_\eps(s)$
converges to a Brownian motion, so that the  marginals of   ${\cal Q}_{x}^\eps$  form a tight
family of measures on ${\cal C}_d$, thus ${\cal Q}_{x}^\eps$ is also a
tight family.  {In
light of~\eqref{022704} and Proposition \ref{prop010601},  the family $\tilde {\cal Q}_{x}^\eps$ of the laws of
$\left(M_\eps^{x}(s),M_\eps(s)\right)_{s\ge 0}$ are also tight, as~$\eps\downarrow 0$, and the families ${\cal
  Q}_{x}^\eps$ and $\tilde{\cal
  Q}_{x}^\eps$  have the same limiting points as $\eps\downarrow 0$, so that we can focus on  $\tilde{\cal
  Q}_{x}^\eps$.}

The processes $\hat B_{j,\eps}^{t,x}(s)$ are square integrable, continuous trajectory
martingales. Thanks to the expressions 
\[
c_{i,j}= \langle A{\mathfrak e}_i,{\mathfrak e}_{j}\rangle_{H},\quad v_{n,m}(x)=\langle A {\mathfrak e}_n^x,{\mathfrak e}_m\rangle_H,
\]
as well as stationarity in space,  their co-variations are
\begin{align}
\label{hat-B1}
\langle \hat B_{j_1,\eps}^{t,x},\hat B_{j_2,\eps}^{t,y}\rangle_s &=
\sum_{k,m=1}^{+\infty}c_{k,m}\int_t^{s}u_{j_1,k}\Big(\frac{X_\eps^{t,x}(\sigma)}{\eps}\Big)u_{j_2,m}
\Big(\frac{X_\eps^{t,y}(\sigma)}{\eps}\Big)d\si\nonumber\\
&
=\int_t^s v_{j_1,j_2}\Big(\frac{X_\eps^{t,x}(\sigma)-X_\eps^{t,y}(\sigma)}{\eps}\Big)d\si,\quad s\ge t,\,x,y\in\bbR^d,
\end{align}
so we have
\begin{align*}
&
\langle M_{\eps,k}^x, M_{\eps,\ell}^0\rangle_s
=\sum_{j,m=1}^{+\infty}\int_0^{s}
 v_{j,m}\Big( \frac{X_\eps^{0,x}(\sigma)-X_\eps(\sigma)}{\eps}\Big) {\cal D}_j\chi_k (\eta_{\eps,\sigma}^{0,x})
  {\cal D}_{m}\chi_\ell
(\eta_{\eps,\sigma})d\sigma
 =\int_0^{s}m_{k,\ell}^{\eps,x,0}(\si)d\si,
 \end{align*}
 with 
 \begin{align}
 \label{042705a}
m_{k,\ell}^{\eps,x,y}(\si):=\left\langle A{\cal D}\chi_k (\eta_{\eps,\sigma}^{0,x}),
  \tau_{[X^{0,x}_\eps(\sigma)-X^{0,y}_\eps(\sigma)]/\eps}{\cal D}\chi_\ell
 \left(\eta_{\eps,\sigma}^{0,y}\right)\right\rangle_{H} ,\,\,k,\ell=1,\ldots,d.
\end{align}

We now perform a finite-dimensional approximation: given $N\in{\mathbb N}$, let
$
{\cal D}\chi_k^N=\sum_{j=1}^N{\cal D}_j\chi_k{\mathfrak e}_j.
$
As 
$$
\sum_{k=1}^d\int_{\cal E}\|{\cal D}\chi_k\|_H^2d\pi<+\infty,
$$
we have
\begin{equation}
\label{012905}
\lim_{N\to+\infty}\sum_{k=1}^d\int_{\cal E}\|{\cal D}\chi_k-{\cal D}\chi_k^N\|_H^2d\pi
=\lim_{N\to+\infty}\sum_{k=1}^d\sum_{j=N+1}^{+\infty}\int_{\cal E}|{\cal D}_j\chi_k|^2d\pi=0.
\end{equation}
Define $m_{k,\ell,N}^{\eps,x,y}$ by \eqref{042705a}, with ${\cal D}\chi_k$, ${\cal D}\chi_\ell$ 
replaced by ${\cal D}\chi_k^N$, ${\cal D}\chi_\ell^N$ correspondingly. Recall that
\[
\eta^{t,x}_s=\tau_{X^{t,x}(s)}V_s, \quad \eta^{t,x}_{\eps,s}=\tau_{X^{t,x}_\eps(s)/\eps}V_{s/\eps^2},
\]
so we have 
\[
\eta^{t,x}_{\eps,s}=\eta^{t/\eps^2,x/\eps}_{s/\eps^2}.
\]
The following approximation property holds.
\begin{lemma}
\label{lm012105}
For any $(s,x)\in[0,+\infty)\times \bbR^d$, we have
\begin{eqnarray}
\label{022705}
\lim_{N\to+\infty}\sup_{\eps\in(0,1]}\bbE\left|m_{k,\ell}^{\eps,x,0}(s)-m_{k,\ell,N}^{\eps,x,0}(s)\right|=0,\,\,k,\ell=1,\ldots,d.
\end{eqnarray}
\end{lemma}
\proof 
The expression under the limit in \eqref{022705} can be estimated using the Cauchy-Schwarz inequality as 
\begin{align*}
&C\Big\{\bbE\Big\|\left({\cal D}\chi_k-{\cal D}\chi_k^N\right) \big(\eta^{0,x/\eps}_{s/\eps^2}\big)
\Big\|_H^2\Big\}^{1/2}\Big\{\bbE\Big\|
 \tau_{X^{0,x/\eps}(s/\eps^2)-X(s/\eps^2)} {\cal D}\chi_\ell(\eta_{s/\eps^2})\Big\|_H^2\Big\}^{1/2}\\
 &
 +C\Big\{\bbE\Big\| \tau_{X^{0,x/\eps}(s/\eps^2)-X(s/\eps^2)} {\cal D}\chi_k^N(\eta_{s/\eps^2})\Big\|_H^2\Big\}^{1/2}
 \Big\{\bbE\Big\| \left({\cal D}\chi_\ell-{\cal D}\chi_\ell^N\right)(\eta^{0,x/\eps}_{s/\eps^2})\Big\|_H^2\Big\}^{1/2},
\end{align*}
with a constant $C>0$, independent of $\eps>0$ and $N$.
The group $\tau_x$ is unitary on $H$ and the processes~$\eta^{0,x/\eps}_{t/\eps^2} $ are stationary in $t$ for
each $x$ fixed. Therefore, the above expression equals
\begin{align*}
&C\Big\{\!\int_{\cal E}\big\| {\cal D}\chi_k-{\cal D}\chi_k^N
\big\|_H^2d\pi\Big\}^{1/2}\Big\{\!\int_{\cal E}\left\|
 {\cal D}\chi_\ell\right\|_H^2d\pi\Big\}^{1/2}\!\!\!
 +C\Big\{\!\int_{\cal E}\left\|
 {\cal D}\chi_k^N\right\|_H^2d\pi\Big\}^{1/2}\Big\{\!\int_{\cal E}\big\| {\cal D}\chi_\ell-{\cal D}\chi_\ell^N
\big\|_H^2d\pi\Big\}^{1/2}.
\end{align*}
The claim of the lemma can be now concluded directly from \eqref{012905} and Corollary~\ref{cor013101}.
\qed
%
%
%
\bigskip

The next lemma shows that if the trajectories are sufficiently far apart, their co-variation is small. For any measurable set $A\subset \Omega$ and random variable $X$, we write $\E[X, A]=\E[X 1_A]$.
\begin{lemma}
\label{lm022105}
For any $\ga\in(0,1)$, $x\not=0$ and  $s,N>0$ we have
\begin{eqnarray}
\label{022705a}
\lim_{\eps\to0}
\bbE\left[\left|m_{k,\ell,N}^{\eps,x,0}(s)\right|, \,|X^{0,x}_\eps(s)-X_\eps(s)|\ge \eps^{\ga}\right]=0,\,\,k,\ell=1,\ldots,d.
\end{eqnarray}
\end{lemma}
\proof  We write (\ref{042705a}) as
\begin{align}
 \label{032905}
m_{k,\ell,N}^{\eps,x,0}(s)=\sum_{j,p=1}^N{\cal D}_j\chi_k \big(\eta^{0,x/\eps}_{s/\eps^2}\big) 
{\cal D}_{p}\chi_\ell \big(\eta_{s/\eps^2}\big)v_{j,p} \Big(\frac{X^{0,x}_\eps(s)-X_\eps(s)}{\eps}\Big),
\end{align}
and estimate
\begin{eqnarray*}
&&
\bbE\big[|m_{k,\ell}^{\eps,x,N}(s)|,
  \,|X^{0,x}_\eps(s)-X_\eps(s)|\ge \eps^{\ga}\big]
\le
\sum_{j,p=1}^N \sup_{|x|\ge \eps^{\ga-1}}|v_{j,p}(x)|
\bbE\big|{\cal D}_j\chi_k (\eta^{0,x/\eps}_{s/\eps^2}) {\cal D}_{p}\chi_\ell
(\eta_{s/\eps^2})\big|\\
&&
\le
\sum_{j,p=1}^N \sup_{|x|\ge \eps^{\ga-1}}|v_{j,p}(x)|
\|{\cal D}_j\chi_k\|_{L^2(\pi)}\|{\cal D}_{p}\chi_\ell\|_{L^2(\pi)}.\nonumber
\end{eqnarray*}
Now, the conclusion of the lemma follows from Proposition \ref{prop012705} since $N$ is finite.
\qed

\subsubsection*{The concatenated process} 
 
Let ${\cal M}_s$ be the natural filtration corresponding
to the canonical process $(X(s),Y(s))_{s\ge0}$ on ${\cal C}_{2d}$, 
and~${\cal M}$ be the smallest $\si$-algebra  generated by all ${\cal M}_s$, $s\ge0$. 
Fix $\ga\in(0,1)$, and for any
 $(X,Y)\in {\cal C}_{2d}$ and $\eps\ge0$ let
$$
T_\eps(X,Y):=\inf\left\{s\ge0:\,|X(s)-Y(s)|\le \eps^{\ga}\right\}.
$$
We adopt the convention that the infimum of an empty set equals $+\infty$. Let us modify the processes~$M_\eps^x(s)$ 
and~$M_\eps(s)$ as follows:
$$
\tilde M^x_\eps(s):=\left\{
\begin{array}{ll}
M_\eps^{x}(s),& 0\le s\le  \tilde T_\eps ,
\\
M_\eps^{x}(\tilde T_\eps)+\beta_{s-\tilde T_\eps},&  \tilde T_\eps\le s,
\end{array}
\right.
$$ 
$$
\tilde M_\eps(s):=\left\{
\begin{array}{ll}
M_\eps(s),& 0\le s\le  \tilde T_\eps ,
\\
M_\eps(\tilde T_\eps)+\tilde \beta_{s-\tilde T_\eps},&  \tilde T_\eps\le s.
\end{array}
\right.
$$ 
Here $\tilde T_\eps :=T_\eps( X^x_\eps , X_\eps )$, and  $\beta_s$ and $\tilde \beta_s$  are two copies of the Brownian
motion  with the covariance given by \eqref{cov-B} that are independent
of each other  and of $(X_\eps^{x}(s),X_\eps(s))_{s\ge 0}$. 
We denote by
$\hat{\cal Q}^{\eps}_{x}$
the law of $(\tilde M^x_\eps(s), \tilde M_\eps(s))_{s\ge
  0}$ on $({\cal C}_{2d},{\cal M})$, and   the law of $(x+\beta_s,y+\tilde \beta_s)$ by $\mathfrak Q_{x,y}$.
The following proposition  shows
that the law $\hat{\cal Q}^{\eps}_{x}$  becomes close to  $\mathfrak
Q_{x,0}$,  as $\eps\to 0$.  To abbreviate the notation, we set
\[
N_t(G):=G(X(t),Y(t))- G(X(0),Y(0))-
\int\limits_0^t(A_x+A_y)G(X(\varrho),Y(\varrho)))\,d\varrho
\]
for any $G\in C^{2}_0(\bbR^{2d})$ and $t\ge0$. Here $A_x$, $A_y$ denote the
differential operators of the form 
\begin{equation}
\label{033005}
AG(x):=\frac{1}{2}\sum_{k,\ell=1}^{d}a_{k\ell}\partial^2_{x_k,x_\ell}G(x),
\end{equation}
acting on the $x$ and
$y$ variables respectively.
\begin{proposition}
\label{prop706041}
For any $x\in\bbR^d$, the family of laws $(\hat{\cal Q}^{\eps}_{x})_{\eps\in(0,1]}$ is tight. 
Suppose, in addition, that~$x\not=0$, $\zeta\in
C_b((\bbR^{2d})^{n})$, and
$0\leq t_1<\cdots< t_n\le t<v\le T$.
Then, we have
\begin{equation}
\label{73101}
\lim_{\eps\to0}E^{\eps}_{x}\left\{\left[N_v(G)-N_t(G)\right] \tilde\zeta\right\}=0
\end{equation}
for any $G\in C^{2}_0(\bbR^{2d})$.
Here $E^{\eps}_{x}$ denotes the expectation with respect to $\hat{\cal Q}^{\eps}_{x}$, and
$$
\tilde\zeta(X,Y):=\zeta(X(t_1),Y(t_1),\ldots,
X(t_n),Y(t_n)), \quad (X,Y)\in{\cal C}_{2d}.
$$
\end{proposition}
\proof
Tightness  is a direct consequence of the tightness of
$\tilde{{\cal Q}}^{\eps}_{x}$, $\eps\in(0,1]$, so we only need to show~(\ref{73101}). 
Denote
$$
\hat  m_{k,\ell}^{\eps,x}(s):=\left\{
\begin{array}{ll}
 m_{k,\ell}^{\eps,x,x}(s),&s\le \tilde T_\eps,\\
&\\
a_{k\ell},& \tilde T_\eps<s,
\end{array}
\right.~~~~~
\tilde  m_{k,\ell}^{\eps}(s):=\left\{
\begin{array}{ll}
 m_{k,\ell}^{\eps,x,0}(s),&s\le \tilde T_\eps,\\
&\\
0,& \tilde T_\eps<s,
\end{array}
\right.
$$
where $m_{k,\ell}^{\eps,x,y}$ were defined in \eqref{042705a}. Using the It\^o formula, we conclude that
 {\[
{\cal N}_t^\eps(G):=G(\tilde M_\eps^{x}(t),\tilde M_\eps(t))- G(\tilde{M}_\eps^{x}(0),\tilde{M}_\eps(0))-\int_0^t(A_\eps^{x}(s)G+A_\eps^{x,0}(s)G+A_\eps^{0}(s)G)(\tilde M_\eps^{x}(s),\tilde M_\eps(s))ds
\]}
is a martingale, where
$$
A_\eps^{x}(s)G(x,y):=\frac{1}{2}\sum_{k,\ell=1}^{d}\hat m_{k,\ell}^{\eps,x}(s)\partial^2_{x_k,x_\ell}G(x,y),
$$
$$
A_\eps^{0}(s)G(x,y):=\frac{1}{2}\sum_{k,\ell=1}^{d}\hat m_{k,\ell}^{\eps,0}(s)\partial^2_{y_k,y_\ell}G(x,y),
$$
and
$$
A_\eps^{x,0}(s)G(x,y):=\sum_{k,\ell=1}^{d}\tilde m_{k,\ell}^{\eps}(s)\partial^2_{x_k,y_\ell}G(x,y).
$$
Let
$$
\tilde\zeta_\ep:=\zeta(\tilde M_\eps^{x},\tilde M_\eps)\quad \mbox{and}\quad  
 \tilde\zeta_\eps'=\zeta(X_\eps^{0,x},X_\eps).
$$
Since
\begin{align*}
&\lim_{\eps\to0}E^{\eps}_{x}\left\{\left[N_v(G)-N_t(G)\right]
  \tilde\zeta\right\}\\
&
=
\lim_{\eps\to0}\left\{\vphantom{\int_0^1}\bbE \left\{\left[N_v(G; \tilde M_\eps^{x},\tilde M_\eps)-N_t(G; \tilde M_\eps^{x},\tilde M_\eps)\right]
  \tilde\zeta_\eps\right\}-\bbE\left\{[{\cal N}_v^\eps(G)-{\cal N}_t^\eps(G)]\tilde\zeta_\eps\right\}\right\},
\end{align*}
to prove \eqref{73101}, it suffices to show that
\begin{align}
\label{013005}
&
\lim_{\eps\to0}\bbE\left\{\left[\int_t^v\left(A_\eps^{x}(\si)-A_x\right)G(\tilde M_\eps^{x}(\si),\tilde M_\eps(\si))d\si\right]\tilde\zeta_\eps\right\}=0,\nonumber\\
&
\lim_{\eps\to0}\bbE\left\{\left[\int_t^v\left(A_\eps^{0}(\si)-A_y\right)G(\tilde M_\eps^{x}(\si),\tilde M_\eps(\si))d\si\right]\tilde\zeta_\eps\right\}=0,\\
&
\lim_{\eps\to0}\bbE\left\{\left[\int_t^vA_\eps^{x,0}(\si)G(\tilde M_\eps^{x}(\si),\tilde M_\eps(\si))d\si\right]\tilde\zeta_\eps\right\}=0.\nonumber
\end{align}
 {Choose an arbitrary $\delta>0$ and integers $N_1,N_2>0$, and divide the interval $[t,v]$ into
subintervals~$[s_{k-1},s_k]$, with $s_k:=t+k(v-t)/N_1$,
$k=0,\ldots,N_1$. As the laws of~$(\tilde M^{x}_\eps,\tilde M_\eps)$ are tight,
we can choose $N_1,N_2$ sufficiently large  so that the limit of the first
expression in \eqref{013005} differs only by $\delta$ from  
\begin{align}
\label{023005}
&
\lim_{\eps\to0}\sum_{j=1}^{N_1}\bbE\left\{\left[\int_{s_{j-1}\wedge \tilde T_\eps}^{s_j\wedge \tilde T_\ep}w_{k,\ell}^{(N_2)}
    (\eta^{s_{j-1},X^{0,x}_\eps(s_{j-1})/\eps}_{\eps,\si})\partial^2_{x_k,x_\ell}G(\tilde
    M_\eps^{x}(s_{j-1}),\tilde M_\eps(s_{j-1}))d\si\right]
    \tilde\zeta_\eps\right\},
\end{align}
where 
$$
w_{k,\ell}^{(N_2)}:=\langle A{\cal D}{\frak P}_{0,N_2} \chi_k,{\cal D}
{\frak P}_{0,N_2} \chi_\ell\rangle_H-\int_{\cal E}\langle A{\cal D}{\frak P}_{0,N_2} \chi_k,{\cal D} {\frak P}_{0,N_2} \chi_\ell\rangle_Hd\pi
$$
and ${\frak P}_{0,N_2}:=\sum_{n=0}^{N_2}{\frak p}_n$. Clearly $w_{k,\ell}^{(N_2)}\in{\frak H}_{\infty}$.
Let 
$\theta^{(N_2)}_{k,\ell}\in{\frak H}_{\infty}$ be the mean-zero solutions of 
\begin{equation}
\label{cor2}
- {\cal L}\theta^{(N_2)}_{k,\ell}=w_{k,\ell}^{(N_2)},\quad k,\ell=1,\ldots,d,
\end{equation}
that exist, thanks to Theorem \ref{prop013101}.
Using formula \eqref{f1} we get
\begin{align*}
&
d\theta^{(N_2)}_{k,\ell}(\eta_{\eps,\si}^{s,x})
 =-\frac{1}{\eps^2}
 {\cal
   L}\theta^{(N_2)}_{k,\ell}(\eta_{\eps,\si}^{s,x})
 d\sigma+
 \frac{1}{\eps}\Big\langle  \tau_{-X_\eps^{s,x}(\si)/\eps}{\cal D} \theta^{(N_2)}_{k,\ell}(\eta_{\eps,\si}^{s,x}),dB^\eps_\si\Big\rangle_H,
\end{align*}
with $B^\eps_\si:=\eps B_{\si/\eps^2}$. Substituting from  the above into \eqref{023005},  
we conclude  that  
$$
\bbE\left\{\left[\int_{s_{j-1}\wedge\tilde T_\eps}^{s_j \wedge\tilde T_\eps}{\cal
      L}
    \theta^{(N_2)}_{k,\ell}(\eta^{s_{j-1}/\eps^2,X^{0,x/\eps}(s_{j-1}/\eps^2)}_{\si/\eps^2})\partial^2_{x_k,x_\ell}G(\tilde
    M_\eps^{x}(s_{j-1}),\tilde M_\eps(s_{j-1}))d\si\right]\tilde\zeta_\eps \right\}=O(\eps^2).
$$
It follows that  
$$\lim_{\eps\to0}\sum_{j=1}^{N_1}\bbE\left\{\left[\int_{s_{j-1}\wedge
      \tilde T_\eps}^{s_j \wedge
      \tilde T_\eps}{\cal
      L}
    \theta^{(N_2)}_{k,\ell}(\eta^{s_{j-1}/\eps^2,X^{0,x/\eps}(s_{j-1}/\eps^2)}_{\si/\eps^2})\partial^2_{x_k,x_\ell}
    G(\tilde
    M_\eps^{x}(s_{j-1}),\tilde M_\eps(s_{j-1}))d\si\right]\tilde\zeta_\eps \right\}=0, 
$$ 
for any $N_1,N_2$ fixed, and the first equality in \eqref{013005} follows. The
second equality can be obtained in the same way. The third equality is
then a direct consequence of Lemma~\ref{lm012105} and  \ref{lm022105}.}
\qed

 \bigskip

It follows from Proposition \ref{prop706041} that  $\hat{\cal
    Q}^{\eps}_{x}$ converge weakly, as
$\eps\to 0$ to ${\mathfrak Q}_{x,0}$. 
We also have
\begin{equation}\label{nov2802}
\tilde{{\cal
    Q}}^{\eps}_{x}(A,T_{\eps}> T)=\hat{\cal
    Q}^{\eps}_{x}(A,T_{\eps}> T),\quad A\in{\cal
  M}_T, T>0,\,\eps\in(0,1].
\end{equation}
Therefore, for any $0<\eps<\eps'\le 1$ we have
\begin{equation}\label{nov2804}
 \tilde{{\cal
    Q}}^{\eps}_{x}(T_{\eps}\le T)=\hat{\cal
    Q}^{\eps}_{x}(T_{\eps}\le T)\le \hat{\cal
    Q}^{\eps}_{x}(T_{\eps'}\le T).
\end{equation}
Passing to the limit, as $\eps\to 0$, and using elementary properties of
weak convergence of probability measures, we see that 
\begin{equation}\label{nov2806}
\limsup_{\eps\to0}\hat{\cal
    Q}^{\eps}_{x}(T_{\eps'}\le T)\le {\mathfrak Q}_{x,0}(T_{\eps'}\le T),\quad \eps'\in(0,1].
\end{equation}
The last point is that, as $\beta_t$ and $\tilde\beta_t$ are two independent Brownian motions
with non-degenerate covariances,  and $d\ge 2$, we have
\begin{equation}
\label{010306b}
{\mathfrak Q}_{x,0}(T_0< T)=0,\quad \mbox{for any }T>0,\,x\not=0.
\end{equation}
The weak convergence of  $\hat{\cal Q}^{\eps}_{x}$ to ${\mathfrak Q}_{x,0}$ and
(\ref{nov2802})-(\ref{010306b}) imply
  the conclusion of Theorem~\ref{thm011105}.\qed

\subsection*{Proof of Corollary \ref{thm-c}}


It suffices to show that 
\begin{equation}
\label{020902}
\lim_{\eps\to0}\bbE \langle u_\eps(t),\varphi\rangle=\langle \bar
u(t),\varphi\rangle\quad\mbox{and}\quad  \lim_{\eps\to0}\bbE \langle u_\eps(t),\varphi\rangle^2=\langle \bar
u(t),\varphi\rangle^2.
\end{equation}
The first equality follows from the weak convergence of $u_\eps(t,x)=u_0\left(X^{t,x}_\eps(T)\right)$
to $u_0(x+\beta_{T-t})$. To prove the second equality, observe that
for any $\varphi\in C_0^\infty(\bbR^d)$
\begin{align*}
&\lim_{\eps\to0}\bbE \langle
u_\eps(t),\varphi\rangle^2=\int_{\bbR^{2d}}\left\{\lim_{\eps\to0}\bbE\left[u_0\left(X^{t,x}_\eps(T)\right)
  u_0\left(X^{t,y}_\eps(T)\right)\right]\right\}\varphi(x)\varphi(y)dx dy.
\end{align*}
Using Theorem \ref{thm011105} we conclude that the right side equals
$$
\int_{\bbR^{2d}}\bbE\left[u_0(x+\beta_{T-t})
  u_0(y+\tilde\beta_{T-t})\right]\varphi(x)\varphi(y)dx dy=\langle \bar
u(t),\varphi\rangle^2.
$$\qed


\section{Proofs of Theorem \ref{thm-semilin} and Corollary  \ref{thm-semilin-c}}

\label{sec9}

Let $\tilde f:D_T\times {\cal E}\to\bbR$ be given by
$$
\tilde f(t,x,u,w):=f(t,x,u,w)-\bar f(t,x,u),
$$
where $\bar f$ is defined by \eqref{021906},
and
$\Theta:D_T\times {\cal E}\to\bbR$ be the unique solution of 
\begin{equation}
\label{031906}
-{\cal L}\Theta(t,x,u,w)=\tilde f(t,x,u,w),\quad \int_{\cal E}\Theta(t,x,u,w)d\pi=0,\quad (t,x,u)\in D_T.
\end{equation}
We note that $(t,x,u)\in D_T$ is fixed in \eqref{031906}, and the ${\cal L}$ operator is acting only on the variable $w\in {\cal E}$. To simplify the notation, we will keep the dependence on $w$ implicit. By the It\^o formula (\ref{f1}), 
\begin{equation}
\label{012006}
\int_{s_1}^{s_2}\tilde
f(s,y,u,\eta^{t,x}_{\eps,\si})d\si=\eps^2\Delta_{s_1,s_2}^\eps\Theta(s,y,u)+\eps
\Delta_{s_1,s_2}^\eps M_\eps(s,y,u),
\end{equation}
where
$$
\Delta_{s_1,s_2}^\eps\Theta(s,y,u):=\Theta(s,y,u,\eta^{t,x}_{\eps,s_1})-\Theta(s,y,u,\eta^{t,x}_{\eps,s_2}),
$$
$$
\Delta_{s_1,s_2}^\eps M(s,y,u):=\int_{s_1}^{s_2}\langle
\tau_{-X^{t,x}_\eps(\si)}{\cal D} \Theta(s,y,u,\eta^{t,x}_{\eps,\si}),dB^\eps_{\si}\rangle_H,\quad
t\le s, s_1,s_2\le T,  \,s_1<s_2,
$$
and $B^\eps_t:=\eps B(t/\eps^2)$. To simplify the notation, we will omit the dependence of $\Theta$ on the $w$ variable. Given any $N>0$, we let 
\begin{equation}
\label{KM}
K_N:=[(y,u)\in \mathbb{R}^{d+1}:\,\max\{|y|,\,|u|\}\le N].
\end{equation}
%

\begin{lemma}
\label{lm022006}
For any $N>0$, $p\in(1,+\infty)$ and $t\le s_1<s_2\le T$, we have
\begin{equation}
\label{042006}
\sup_{s\in[t,T]}\Big\langle\sup\limits_{(y,u)\in K_N}\left|\Theta(s,y,u)\right|^p\Big\rangle_{\pi}<+\infty,
\end{equation}
and
\begin{equation}
\label{052006}
\limsup_{\eps\to 0}\sup_{s\in[t,T]}\Big\langle \sup\limits_{(y,u)\in K_N}|\Delta_{s_1,s_2}^\eps M(s,y,u)|^2\Big\rangle_\pi<+\infty.
\end{equation}
\end{lemma}
\proof  Given an integer multi-index $m=(m_0,\ldots,m_d)$, we denote by $|m|=\sum_{j=0}^d m_j$, and
by $\nabla^m$  
the mixed partial $\partial_u^{m_0}\partial_{y_1}^{m_1}\ldots\partial_{y_d}^{m_d}$.
Then $\nabla^m\Theta(t,y,u)$ is the unique 
$\pi$-zero-mean solution  of  
\begin{equation}
\label{031906m}
-{\cal L}\nabla^m\Theta(t,y,u)=\nabla^m\tilde f(t,y,u).
\end{equation}
In addition, thanks to \eqref{sg1a} and \eqref{sg1b}, for any $p\in(1,+\infty)$  there exists $C>0$ such that
\begin{equation}
\label{012306}
\|\nabla^m\Theta(t,y,u)\|_{L^p(\pi)}\le C \|\nabla^m\tilde
f(t,y,u)\|_{L^p(\pi)},\quad (t,y,u)\in D_T.
\end{equation}
The Sobolev embedding theorem implies that for $p>d+1$ there exists a constant $C>0$ such that
\begin{equation}
\label{012206a}
\sup\limits_{(y,u)\in K_N}|\Theta(t,y,u)|^p\le C\int_{K_N}\Big[|\Theta(t,y,u)|^p+\sum_{|m|=1}
\left|\nabla^m\Theta(t,y,u)\right|^p\Big]dy du.
\end{equation}
Integrating both sides of \eqref{012206a} in the $w$ variable and using \eqref{012306}, we obtain  \eqref{042006}.

Next, we proceed with the proof of \eqref{052006}. It is clear from the definition of the Malliavin derivative, see
\eqref{Dj} and \eqref{D}, that
$$
{\cal D}\nabla^m\Theta(t,y,u)=\nabla^m {\cal D}\Theta(t,y,u),
$$
and
$$
\nabla^m \big(\Delta_{s_1,s_2}^\eps M\big)(s,y,u)=\int_{s_1}^{s_2}\langle
\tau_{-X^{t,x}_\eps(\si)}\nabla^m{\cal D} \Theta(s,y,u;\eta^{t,x}_{\eps,\si}),dB^\eps_\si\rangle_H,\quad
t\le s, s_1,s_2\le T.
$$ 
Again, thanks to the Sobolev embedding theorem, there exists a constant $C>0$ such that
\begin{equation}
\label{012206}
\sup\limits_{(y,u)\in K_N}\big[\Delta_{s_1,s_2}^\eps M(s,y,u)\big]^2\le 
C\int_{K_N}\Big\{\big[\Delta_{s_1,s_2}^\eps M(s,y,u)\big]^2+
\!\sum_{|m|=n}\big[\nabla^m(\Delta_{s_1,s_2}^\eps M)(s,y,u)\big]^2\Big\}dy du,
\end{equation}
provided that $n>(d+1)/2$. Using the It\^o isometry and the above
estimate, we conclude that for any $s_1<s_2$ 
\begin{eqnarray}
\label{022206}
&&\!\!\!\!\!\!\!\!\!\!\!\!\bbE\Big\{\!\sup\limits_{(y,u)\in K_N}\big[\Delta_{s_1,s_2}^\eps M(s,y,u)\big]^2\Big\}
\le C(s_2-s_1)\int_{K_N}\!\!\!\Big\langle \left\|{\cal D}\Theta(s,y,u)\right\|_H^2+\!\!\sum_{|m|=n}
\left\|{\cal D}\nabla^m\Theta(s,y,u)\right\|^2_H\Big\rangle_{\pi}dy du\nonumber\\
&&
\leq C(s_2-s_1)\int_{K_N}\Big[{\cal E}_{\cal L}\left(\Theta(s,y,u)\right)+
\sum_{|m|=n}{\cal E}_{\cal L}\left(\nabla^m\Theta(s,y,u)\right)\Big]dy du. 
\end{eqnarray}
This proves \eqref{052006} in light of  (\ref{012306}) and~\eqref{H1}.
\qed

\bigskip

Next, let
\begin{equation}
\label{020107}
U_\eps(s):=u_\eps(s,X_\eps^{t,x}(s)),\quad F_\eps(s):=f(s,X_\eps^{t,x}(s),U_\eps(s),\eta_{\eps,s}^{t,x})
\end{equation}
and
$
\tilde F_\eps(s):=\tilde
f(s,X_\eps^{t,x}(s),U_\eps(s),\eta_{\eps,s}^{t,x})$,  $t\le s \le T.
$
We have
\[
u_\eps(t,x)=U_\eps(t)
\]
and
\begin{equation}
\label{011906}
U_\eps(T)=U_\eps(s)+\int_s^TF_\eps(\si)d\si,
\end{equation}
so that
\begin{equation}
\label{011906a}
U_\eps(T)-U_\eps(s)-\int_s^T\bar f(\si,X_\eps^{t,x}(\si),U_\eps(\si))d\si=\int_s^T\tilde F_\eps(\si)d\si.
\end{equation}
The following lemma shows that the random fluctuation in the r.h.s. of the above display is negligible in the limit.
\begin{proposition}
\label{cor012006}
For any $\delta>0$, $t\le s\le T$  we have
\begin{equation}
\label{062006}
\lim_{\eps\to0}\bbP\Big[\Big|\int_{s}^{T}\tilde F_\eps(\si)d\si\Big|\ge \delta\Big]=0.
\end{equation}
\end{proposition}
\begin{proof}  
Since $f$ is bounded,   the laws of the processes $\left( U_\eps(s)\right)_{t\leq s\le T}$
are tight over $C[t,T]$, as $\eps\to 0$.
 In consequence, the laws of the  joint process $\left( X_\eps^{t,x}(s) ,U_\eps(s)\right)_{t\leq s\le T}$
are also tight.
Given any $\rho>0$, one
can choose~$N>0$ such that 
\begin{equation}
\label{042206}
 \bbP\Big[\sup_{s\in[t,T]}\max[|X_\eps^{t,x}(s)|, |U_\eps(s)|]\ge
   N\Big]<\rho,\quad \eps\in(0,\eps_0].
\end{equation}
Thanks to \eqref{042206}, we can find a sufficiently large $N$ so that
\begin{equation}
\label{062006a}
\limsup_{\eps\to 0}\bbP\Big[\Big|\int_{s}^{T}\tilde F_\eps(\si)d\si\Big|\ge \delta,\,\sup_{s\in[t,T]}\max[|X_\eps^{t,x}(s)|, |U_\eps(s)|]\ge
   N\Big]<\frac{\rho}{3}.
\end{equation}
Let $M$ be a non-negative integer and  $t_j:=s+j(T-s)/M$, $j=0,\ldots,M$. Using the tightness of $\left( X_\eps^{t,x}(s) ,U_\eps(s)\right)_{t\leq s\le T}$, we can choose a sufficiently large $M_0$ so that
\begin{equation}
\label{062006b}
\limsup_{\eps\to 0}\sup_{M\ge M_0}\bbP\Big[\int_{s}^{T}\left|\tilde F_\eps(\si)-\tilde F_{M,\eps}(\si)\right|d\si\ge \delta\Big]<\frac{\rho}{3},
\end{equation}
where
$$
\tilde F_{M,\eps}(s):=\tilde f\big(t_j,X_\eps^{t,x}(t_j),U_\eps(t_j),\eta_\eps^{t,x}(s)\big),\quad t_j\le s <t_{j+1},\,j=0,\ldots,M-1.
$$
To prove \eqref{062006}, it suffices to show that given  $M$, $N$ and $\rho>0$ we have 
\begin{equation}
\label{062006c}
\limsup_{\eps\to 0}\bbP\Big[\Big|\int_{s}^{T}\tilde F_{M,\eps}(\si)d\si\Big|\ge \delta,\,\sup_{s\in[t,T]}\max[|X_\eps^{t,x}(s)|, |U_\eps(s)|]\le
   N\Big]<\frac{\rho}{3}.
\end{equation}
Obviously, we have
$$
\Big|\int_{s}^{T}\tilde F_{M,\eps}(\si)d\si\Big|\le \sum_{j=1}^{M-1}\Big|\int_{t_j}^{t_{j+1}}\tilde f\left(t_j,X_\eps^{t,x}(t_j),U_\eps(t_j),\eta_{\eps,\si}^{t,x}\right)d\si\Big|.
$$
Estimate \eqref{062006c} holds, provided we prove that for any
$N>0$ and $t\le s\le s'\le T$:
\begin{equation}
\label{062006d}
\limsup_{\eps\to 0}\bbE\Big[\sup_{(y,u)\in K_N}\Big|\int_{s}^{s'}\tilde f\left(s,y,u,\eta_{\eps,\si}^{t,x}\right)d\si\Big|\Big]=0.
\end{equation}
The latter however is a direct consequence of the decomposition
\eqref{012006} and Lemma \ref{lm022006}.
\end{proof}

\bigskip

Given $ X\in C([t,T];\bbR^d)$ we let $U:=\Phi(X)\in C[t,T]$ be the unique solution of  
\begin{equation}\label{072006}
u_0(X(T))=U(s)+\int_s^T\bar f(\si,X(\si),U(\si))d\si,\quad
s\in[t,T].
\end{equation}
Suppose that ${\mathfrak Q}_\eps$ are the laws of $\left( X_\eps^{t,x}(s)
  ,U_\eps(s)\right)_{s\le T}$ over $C([t,T];\bbR^{d+1})$ for $\eps\in(0,1]$. Let
${\mathfrak Q}_*$ be the limiting law of ${\mathfrak Q}_{\eps_n}$ for some sequence $\eps_n\to 0$. Thanks to Proposition \ref{cor012006}, we know that ${\mathfrak Q}_*$ is supported on the set 
$$
{\cal C}:=\{(X,U): X\in C([t,T];\bbR^d),\,U=\Phi(X)\}.
$$
 As $\left(X_\eps^{t,x}(s)\right)_{s\in[t,T]}$  converges in law to
$\left(x+\beta_{s-t}\right)_{s\in[t,T]}$, as $\eps\to 0$, and $u_\eps(t,x)=U_\eps(t)$, we know that for fixed $(t,x)$, $u_\eps(t,x)$ converges in distribution to $\mathcal{U}(t;t,x)$ with $\mathcal{U}(s;t,x)$ solving 
\begin{equation}\label{e.12311}
u_0\left(x+\beta_{T-t}\right)-\mathcal{U}(s;t,x)= \int_s^T\bar f\left(\si,x+\beta_{\si-t},\mathcal{U}(\si;t,x)\right)d\si,\quad t\le s\le T.
\end{equation}

Concerning the convergence of the multi-point statistics the argument
from Section \ref{sec4.2} proves that for $N$ distinct points
$x_1,\ldots,x_N\in\bbR^d$ the process $(X_\eps^{t,x_1}(s),\ldots,
X_\eps^{t,x_N}(s))_{s\ge t}$ converges in law to
$ (x_1+\beta_{s-t}^{(1)},\ldots, x_1+\beta_{s-t}^{(N)})_{s\ge t}$, where $(\beta_{s}^{(j)})_{s\ge0}$,
$j=1,\ldots,N$ are i.i.d. Brownian motions with the covariance \eqref{apq}.
This, in turn implies  that  $(\mathcal{U}^{(1)}
(t,x_1),\ldots,\mathcal{U}^{(N)} (t,x_n))$, the respective limit 
of~$(u_\eps(t,x_1) ,\ldots,u_\eps(t,x_N))$ is determined 
by the solutions of \eqref{e.12311} based on
$(\beta_{s}^{(j)})_{s\ge0}$, $j=1,\ldots,N$, thus they are
independent. This ends the proof of Theorem \ref{thm-semilin}.\qed

\bigskip

The proof of Corollary \ref{thm-semilin-c} follows essentially from the same
argument as Corollary \ref{thm-c}. It suffices only to note that from
\eqref{011906} it follows that
$\|u_\eps(t,\cdot)\|_{L^\infty(\bbR^d)}$, is deterministically bounded
for $\eps\in(0,1)$. Thus the
random variables $\langle u_\eps(t),\varphi\rangle$ are also
deterministically bounded , for any test function $\varphi\in
L^1(\bbR^d)$. We can repeat then the argument used to show
\eqref{020902} to conclude \eqref{020901}.\qed



\section{Proof of Theorem \ref{thm011110}}
\label{sec8.2}

\subsection{The limiting dynamics and the
proof of Proposition \ref{prop010507}}
\label{sec7.1}

Let us first explain how the coefficients $b$ and $\tilde c_j$ 
in (\ref{020307}) are defined. We set
\begin{align}
\label{b}
& b\left(s,y,u\right):=\left\langle\nabla_y\Theta\left(s,y,u,\cdot\right)\cdot
v(\cdot)\right\rangle_{\pi}+
\left\langle\partial_u \Theta\left(s,y,u,\cdot\right) 
 f\left(s,y,u,\cdot\right)\right\rangle_{\pi},
\end{align}
where $\Theta:D_T\times{\cal E}\to\bbR$ is the solution of
\eqref{031906}, with $\tilde f$ in the right side replaced by $f$ -- recall that
now we assume $f$ has mean zero.

In order to define $\tilde c_j$, recall that the constant matrix $S=a^{1/2}$, with
$a_{ij}$ given by~\eqref{apq}: 
\[
a_{ij} :={\cal E}_L(\chi_i,\chi_{j}),~~i,j=1,\ldots,d.
\]
In particular, $a$ is non-singular.
Then $\tilde c_0(s,y,u)\ge 0$ and
$\tilde c^T(s,y,u)=[\tilde c_1(s,y,u),\ldots,\tilde c_d(s,y,u)]$ are determined by 
\begin{equation}
\label{tc}
S\tilde c=c,\quad \tilde c_0^2(s,y,u):=c_0(s,y,u)-\sum_{j=1}^d\tilde c^2_j(s,y,u)
\end{equation}
with $ c^T=[ c_1, \ldots,c_d]$ and $c_0$ given by
\begin{align}
\label{010609}
c_0(s,y,u)={\cal E}_L(\Theta(s,y,u)),~~~~
c_{j}(s,y,u) :={\cal E}_L(\Theta(s,y,u),\chi_j),\quad j=1,\ldots,d.
\end{align}
We can write
\[
\tilde c_j(s,y,u)=\sum_{k=1}^d S^{-1}_{jk}c_k=
{\cal E}_L(\Theta(s,y,u),\tilde\chi_j),
\]
with the functions
\[
\tilde \chi_j:=\sum_{k=1}^dS^{-1}_{jk}\chi_{k},\quad j=1,\ldots,d,
\]
that are orthonormal with respect to the inner product ${\cal E}_L(\cdot,\cdot)$:
\[
{\cal E}_L(\tilde\chi_j,\tilde\chi_m)=\sum_{k,p=1}^d S_{jk}^{-1}S_{mp}^{-1}
{\cal E}_L(\chi_k,\chi_p)=\sum_{k,p=1}^d S_{jk}^{-1}a_{kp}S_{mp}^{-1}=\delta_{jm}.
\]
Thus, we have
$$
\sum_{j=1}^d\tilde c_j^2(s,y,u)=\sum_{j=1}^d{\cal E}^2_L(\Theta(s,y,u),\tilde \chi_j)\le {\cal E}_L(\Theta(s,y,u))=c_0(s,y,u).
$$

%
%

Let now $(X^{t,x}(s),U^{t,x,u}(s))$, $s\le T$ be the solution of \eqref{020307}
with the above coefficients $b$ and $\tilde c_j$,  
fix $(t,x)\in \bbR^{1+d}$ and $T>t$, and let
$$
\xi_{s}^{t,x}(u):=({\mathfrak
    s}_{s}^{t,x})'(u)=\frac{\partial}{\partial
  u}U^{t,x,u}(s),\quad (s,u)\in[t,T]\times \bbR.  
$$
By the equation satisfied by $U^{t,x,u}(s)$, it is clear that for a fixed $u$, the process $\left(\xi_{s}^{t,x}(u)\right)_{t\le s\le T}$ is a semimartingale that satisfies 
\begin{align}
\label{020307a}
\xi_{s}^{t,x}(u)=1+\int_t^s\al(\si) \xi^{t,x}_{\si}(u) d\si
+\sum_{j=0}^d\int_t^s\ga_j(\si) \xi^{t,x}_{\si} (u)d\tilde\beta_j(\si),
\end{align}
with
$$
\al(\si):=\frac{\partial b}{\partial u}(\si,X^{t,x}(\si),U^{t,x,u}(\si)),
\quad \ga_j(\si):=\pdr{\tilde c_{j}}{u}(\si,X^{t,x}(\si),U^{t,x,u}(\si)).
$$
The unique solution of \eqref{020307a} is given by
$\xi_{s}^{t,x}(u)=\exp\left\{Z(s)\right\}$, $s\in[t,T]$, with
\begin{align}
\label{020307b}
Z(s):=\int_t^s\Big\{\al(\si)-\frac{1}{2}\sum_{j=0}^d\ga_j^2(\sigma)\Big\}d\si
+\sum_{j=0}^d\int_t^s\ga_j(\si) d\tilde\beta_j(\si).
\end{align}
Thus, $\xi_s^{t,x}(u)>0$ a.s., and since for any $s\in[t,T]$  we have
$$
\lim_{u\to\pm\infty}U^{t,x,u}(s)=\pm\infty,\quad \mbox{a.s.},
$$
we conclude from the above that ${\mathfrak
  s}_{s}^{t,x}(u)=U^{t,x,u}(s)$, $u\in\bbR$ is a diffeomorphism. This
ends the proof of Proposition \ref{prop010507}.\qed




\subsection{The truncated dynamics and convergence of the forward process}

Recall that 
we use the notation ${\mathfrak S}_{\eps}^{t,x}(s,u)=U_\eps^{t,x,u}(s)$, where~$U_\eps^{t,x,u}(s)$ is the solution of 
\begin{equation}
\label{U-eps}
U^{t,x,u}_\eps(s) =u+\frac{1}{\eps}\int_t^s f_\eps(\si)d\si,\quad t\le s \le T,
\end{equation}
where we used the simplified notation
\[
f_\eps(\sigma)=f\Big(\sigma,X_\eps^{t,x}(\sigma),U_\eps^{t,x,u}(\sigma),V\Big(\frac{\sigma}{\eps^2},\frac{X_\eps^{t,x}(\sigma)}{\eps}+\cdot\Big)\Big)=f\Big(\sigma,X_\eps^{t,x}(\sigma),U_\eps^{t,x,u}(\sigma),\eta^{t,x}_{\eps,\sigma}\Big).
\]
The mapping ${\mathfrak s}_{s,\eps}^{t,x}(u):={\mathfrak
  S}_{\eps}^{t,x}(s,u)$, 
  is a diffeomorphism of $\bbR$
onto itself for each $(t,x)$ and
$t\le s\le T$. 
Indeed, for each fixed $\eps>0$, the derivative process 
$$
\xi_{s,\eps}^{t,x}(u):=( {\mathfrak s}_{s,\eps}^{t,x})'(u)=
\frac{\partial}{\partial u}U^{t,x,u}_\eps(s)
$$
satisfies the linear equation
\begin{equation}
\label{011906abis}
\xi_{s,\eps}^{t,x}(u) =1+\frac{1}{\eps}\int_t^s 
 {\pdr{f_{\eps}}{u}}(\si)\xi_{\si,\eps}^{t,x}(u)d\si,
\end{equation}
thus $\xi_{s,\eps}^{t,x}(u)>0$ a.s., 
and since $ \lim_{u\to\pm\infty}{\mathfrak s}_{s,\eps}^{t,x}(u)=\pm\infty$ a.s.,
it is  a diffeomorphism. Therefore, for $u_\eps(t,x)$ to satisfy \eqref{advec-11-semi2},  
it is equivalent to
\begin{equation}
\label{010911}
{\mathfrak s}_{T,\eps}^{t,x}\left(u_\eps(t,x)\right) =u_0(X_\eps^{t,x}(T)).
\end{equation}

The first step in the proof of Theorem~\ref{thm011110} is to establish tightness of the family 
$(X_\eps^{t,x}(\cdot),{\mathfrak S}_{\eps}^{t,x}(\cdot))$. However, instead of proving this directly,
we will first prove tightness for a truncated family of processes $(X_\eps^{t,x}(\cdot),{\mathfrak S}_{\eps,M}^{t,x}(\cdot))$ --
note that only the  ${\mathfrak S}_{\eps}^{t,x}$ component is truncated -- and identify the corresponding limit as $\eps\to 0$.
Then, using the properties of the limit process, we will show that ``truncation does not matter'', and get the limit for the original,
un-truncated process. 
To this end, take $M>1$ and set
$$
f^{(M)}(s,y,v,w):=\phi_M(y,v) f(s,y,v,w),\quad
(s,y,v,w)\in D_T\times{\cal E}.
$$
Here, $\phi_M:\bbR^{1+d}\to[0,1]$ is a smooth cut-off function such that
\[
\hbox{$\phi_M \equiv 1$ on $K_M:=[(y,v):\,|y|\le M,\,|v|\le M]$},
\]
and $\phi_M$ is supported in $ K_{M+2}$,
with $\|\nabla\phi_M\|_\infty\le 1$.
We define  $U^{t,x,u}_{\eps,M}(s)$ as the solution of a modified equation
\eqref{011906a}: 
\begin{equation}
\label{011906amodif}
U_{\eps,M}^{t,x,u}(s)=u+\frac{1}{\eps}\int_t^sf^{(M)}_\eps(\si)d\si,~~~t\le s \le T,
\end{equation}
with 
$$
f^{(M)}_\eps(\si):=f^{(M)}(\si,X_\eps^{t,x}(\si),U_{\eps,M}^{t,x,u}(\si),\eta^\eps_{t,x}(\si)),
$$
where we write 
\[
\eta^\eps_{t,x}(\si)=\eta^{t,x}_{\eps,\sigma}
\]
 to emphasize its dependence on $\sigma$ as a process. We denote by ${\mathfrak S}_{\eps,M}^{t,x}(s,u)$ and
 ${\mathfrak s}_{s,\eps,M}^{t,x}(u)$ the  random field and family
 of diffeomorphisms corresponding to~$U^{t,x,u}_{\eps,M}(s)$, and by $u_{\eps,M}(t,x)$ the unique solution of 
\begin{equation}
\label{010911M}
{\mathfrak s}_{T,\eps,M}^{t,x}\left(u_{\eps,M}(t,x)\right) =u_0(X_\eps^{t,x}(T)).
\end{equation}
To define the limit of the truncated dynamics,
let $U^{t,x,u}_M(s)$ 
be the solution of  the SDE
\begin{align}
\label{020307M}
&
U^{t,x,u}_M(s)=u+\int_t^sb_M(\si,X^{t,x}(\si),U^{t,x,u}_M(\si))d\si+\sum_{j=0}^d\int_t^s\tilde c_{M,j}(\si,X^{t,x}(\si),U^{t,x,u}_M(\si))d\tilde\beta_j(\si),
\end{align}
with  $b_M$ and $\tilde c_M$ as in
\eqref{b}, \eqref{tc} and \eqref{010609} but with $\Theta$ and $f$
replaced by $\Theta_M$ and $f^{(M)}$, respectively. Here, $\Theta_M(s,y,v,w):=\phi_M(y,v) \Theta(s,y,v,w)$ is the solution to
the cell problem \eqref{031906} with $f^{(M)}$ in the
right side. This generates the random field 
$\{{\mathfrak S}_{M}^{t,x}(s,u)=U^{t,x,u}_M(s)\}_{(s,u)\in[t,T]\times \bbR}$
and the corresponding diffeomorphisms  ${\mathfrak
  s}_{T,M}^{t,x}(u):={\mathfrak S}_{M}^{t,x}(T,u)$.  We
let $u_{M}(t,x)$ be the unique solution of 
\begin{equation}
\label{010911MM}
{\mathfrak s}_{T,M}^{t,x}\left(u_{M}(t,x)\right) =u_0(X^{t,x}(T)).
\end{equation}
We will call $({\mathfrak
    S}_{\eps,M}^{t,x}(\cdot),X_\eps^{t,x}(\cdot))$ the ``forward process'', and the goal of this section is
\begin{proposition}
\label{prop012509}
Given $M>1$ and $(t,x)\in [0,T]\times \bbR^{d}$,  the processes $({\mathfrak
    S}_{\eps,M}^{t,x}(\cdot),X_\eps^{t,x}(\cdot))$ converge
weakly, as $\eps\to 0$, over $C([t,T]\times\mathbb{R})\times C([t,T])$, equipped with the standard Frechet metric, to 
$({\mathfrak
    S}_{M}^{t,x}(\cdot),X^{t,x}(\cdot))$.
\end{proposition}

\begin{proof}
We will use the following notation: we set
$
g_{\eps,M}(s,u):=g(s,X_\eps^{t,x}(s),U_{\eps,M}^{t,x,u}(s),\eta^\eps_{t,x}(s))
$
for a given field $g:D_T\times {\cal E}\to\bbR$,  and also {use
$g_{\eps|u}'(s)$ and  $\nabla_xg_{\eps}(s)$ to denote the processes
corresponding to~$g_u:=\partial g/\partial u$ and $\nabla_x g$}, respectively.
Using the It\^o formula for  
\[
\Theta_\eps(s)=\Theta(s,X_\eps^{t,x}(s),U_{\eps}^{t,x,u}(s),\eta^\eps_{t,x}(s)),
\]
and recalling that 
\[
-{\cal L}\Theta(t,x,u,w)=f(t,x,u,w),
\]
we obtain
\begin{align}
\label{f-eps1}
d\Theta_\eps(s)
 =&\left\{\partial_{s}\Theta_\eps(s)+
\frac{1}{\eps}\left[\nabla_x\Theta_\eps\left(s\right)\cdot
v(\eta^\eps_{t,x}(s)) +
\Theta_{\eps|u}'\left(s\right) f_\eps\left(s\right) \right]-\frac{1}{\eps^2}
 f_\eps\left(s\right)
\right\}ds\\
&
+
 \frac{1}{\eps}\left\langle  \tau_{-X^{t,x}_\eps(s)/\eps}{\cal D}\Theta_\eps\left(s\right),dB^\eps_s\right\rangle_H,\nonumber
\end{align}
which, in turn, gives
\begin{align}
\label{f-eps}
&
\frac{1}{\eps}
 f_\eps\left(s\right)ds
 =-\eps d\Theta_\eps\left(s\right)+\left\{\vphantom{\int_0^1}\eps\partial_{s}\Theta_\eps\left(s\right) 
+
\nabla_x\Theta_\eps\left(s\right)\cdot
v(\eta^\eps_{t,x}(s))
+\Theta'_{\eps|u}\left(s\right) f_\eps\left(s\right) 
\vphantom{\int_0^1}\right\}ds\\
&
+
 \left\langle  \tau_{-X^{t,x}_\eps(s)/\eps}{\cal D}\Theta_\eps\left(s\right),dB^\eps_s\right\rangle_H.\nonumber
\end{align}
%
The obvious
analogs of~\eqref{f-eps1} and \eqref{f-eps} for $\Theta_{\eps,M}(s)$ and $f_{\eps}^{(M)}(s)$,
together with (\ref{011906amodif}), lead to a decomposition
\begin{align}
\label{011906b}
&
{\mathfrak S}^{t,x}_{\eps,M}(s,u) =\sum_{j=0}^2{\mathfrak S}^{t,x}_{\eps,M,j}(s,u),
\end{align}
with 
\begin{align}
\label{052909}
&
{\mathfrak S}^{t,x}_{\eps,M,0}(s,u):=\eps \Theta_{\eps,M}\left(t,u\right)-\eps
\Theta_{\eps,M}\left(s,u\right)+\eps \int_t^s\vphantom{\int_0^1}\partial_{\si}\Theta_{\eps,M}\left(\si,u\right) d\si,\nonumber\\
&
{\mathfrak S}^{t,x}_{\eps,M,1}(s,u):= u+\int_t^s\Big\{ 
\nabla_x\Theta_{\eps,M}\left(\si,u\right)\cdot
v(\eta^\eps_{t,x}(\si))
+\Theta'_{\eps,M|u}\left(\si,u\right) f_{\eps,M}\left(\si,u\right) 
 \Big\}d\si,\\
&
{\mathfrak S}^{t,x}_{\eps,M,2}(s,u):=\int_t^s
 \Big\langle  \tau_{-X^{t,x}_\eps(\si)/\eps}{\cal D}\Theta_{\eps,M}\left(\si,u\right),dB^\eps_\si\Big\rangle_H,\quad t\le s \le T.\nonumber
\end{align}

The terms in the right side of (\ref{011906b}) satisfy several estimates given by Lemmas~\ref{lm012509} and \ref{lm022509} below, which we will use to prove the 
tightness of~${\mathfrak S}^{t,x}_{\eps,M}(s,u) $.
%
Take arbitrary $\eta,\eta'>0$ and $N>1$. Since the
right side of \eqref{052909} vanishes for $|u|>M+2$, we may
assume that $N\le M+2$. According to  
Corollary \ref{cor012509} and Lemma \ref{lm022509}, we can choose
$\delta>0$ and~$L>1$ such that $\limsup_{\eps\to 0}\bbP[E_{\eps,L,\delta,\eta'}]<\eta$, where
\begin{equation}
\label{062509}
E_{\eps,L,\delta,\eta'}:=\Big[\mathop{\sup_{t\le
      s<s'\le T,|u|\le N}}_{s'-s<\delta}\left|{\mathfrak
      S}^{t,x}_{\eps,M}(s',u)-{\mathfrak
      S}^{t,x}_{\eps,M}(s,u)\right|\ge\frac{\eta'}{2},\,\mbox{ or }\sup_{s\in[t,T],|u|\le N}|\xi_{s,\eps,M}^{t,x} (u)|>L\Big]
\end{equation}
However, we have
$E_{\eps,L,\delta,\eta'}\supset E_{\eps,\delta',\eta'}$, with $\delta':=\min[\delta,\eta'(2L)^{-1}]$, and
\begin{equation}
\label{062509a}
E_{\eps,\delta',\eta'}:=\Big[\mathop{\sup_{s,s'\in[0,T],|u|,|u'|\le N}}_{|s'-s|+|u-u'|<\delta'}\left|{\mathfrak
      S}^{t,x}_{\eps,M}(s',u')-{\mathfrak
      S}^{t,x}_{\eps,M}(s,u)\right|\ge\eta'\Big].
\end{equation}
Hence, 
$\limsup_{\eps\to 0}\bbP[E_{\eps,\delta',\eta'}]<\eta$ and tightness
follows from  Theorem 2.7.3, p. 82 of \cite{billingsley}.

%
In order to identify the limit, it suffices to prove 
that for any $n\ge1$ and $(s_1',u_1),\ldots,(s_n',u_n)\in[t,T]\times
\bbR$, $s_1,\ldots,s_n\in [t,T]$, we have
\begin{equation}
\begin{aligned}\label{dec402}
&\left({\mathfrak
      S}^{t,x}_{\eps,M}(s_1',u_1),\ldots, {\mathfrak
      S}^{t,x}_{\eps,M}(s_n',u_n),X_\eps^{t,x}(s_1),\ldots,X_\eps^{t,x}(s_n)\right)
\\
&
\stackrel{\eps\to 0}{\Longrightarrow}\left({\mathfrak
      S}^{t,x}_M(s_1',u_1),\ldots,{\mathfrak
      S}^{t,x}_M(s_n',u_n),X^{t,x}(s_1),\ldots,X^{t,x}(s_n)\right).
\end{aligned} 
\end{equation}
To show (\ref{dec402}), we can use \eqref{022704} together with \eqref{011906b}
(recall that $\mathfrak{S}^{t,x}_{\eps,M}(u)=U^{t,x,u}_{\eps,M}(s)$) and
apply a weak convergence argument for semimartingales analogous to
the one used in Section \ref{sec9}. Since the argument is rather similar, we do not present the details. This finishes the proof of Proposition~\ref{prop012509}. \end{proof}

In the following, we present the technical lemmas used in the proof of Proposition~\ref{prop012509}.

\begin{lemma}
\label{lm012509}
For each $M>1$ we have
\begin{equation}
\label{032509}
\lim_{\eps\to 0}\bbE\Big[\sup_{s\in[t,T],|u|\le M+2}\left|{\mathfrak S}^{t,x}_{\eps,M,0}(s,u)\right|\Big]=0.
\end{equation}
In addition, we have
\begin{equation}
\label{042509}
\lim_{\delta\to 0}\limsup_{\eps\to 0}\bbE\Big[\mathop{\sup_{t\le
      s<s'\le T,|u|\le M+2}}_{s'-s<\delta}\left|{\mathfrak
      S}^{t,x}_{\eps,M,1}(s',u)-{\mathfrak
      S}^{t,x}_{\eps,M,1}(s,u)\right|\Big]=0,
\end{equation}
and, for any $\eta>0$ we have
\begin{equation}
\label{042509a}
\lim_{\delta\to 0}\limsup_{\eps\to 0}\bbP\Big[\mathop{\sup_{t\le
      s<s'\le T,|u|\le M+2}}_{s'-s<\delta}\left|{\mathfrak
      S}^{t,x}_{\eps,M,2}(s',u)-{\mathfrak
      S}^{t,x}_{\eps,M,2}(s,u)\right|>\eta\Big]=0.
\end{equation}
\end{lemma}
As a direct corollary, we conclude the following.
\begin{corollary}
\label{cor012509}
For each $M>1$ and $\eta>0$, we have
\begin{equation}
\label{042509abis}
\lim_{\delta\to 0}\limsup_{\eps\to 0}\bbP\Big[\mathop{\sup_{t\le
      s<s'\le T,|u|\le M+2}}_{s'-s<\delta}\left|{\mathfrak
      S}^{t,x}_{\eps,M}(s',u)-{\mathfrak
      S}^{t,x}_{\eps,M}(s,u)\right|>\eta\Big]=0.
\end{equation}
\end{corollary}
We will also need a bound on the derivative process 
$$
\xi_{s,\eps,M}^{t,x}(u):=( {\mathfrak s}_{s,\eps,M}^{t,x})'(u)=\frac{\partial}{\partial u}U^{t,x,u}_{\eps,M}(s),
$$
which 
satisfies an integral equation 
\begin{equation}
\label{011906m}
\xi_{s,\eps,M}^{t,x}(u) =1+\frac{1}{\eps}\int_t^s f'_{\eps,M|u}(\si)\xi_{\si,\eps,M}^{t,x}(u)d\si.
\end{equation}
 \begin{lemma}
\label{lm022509}
 For any $M> 1$, we have
\begin{equation}
\label{011609d}
\lim_{L\to+\infty}\limsup_{\eps\to 0}\bbP\Big[\sup_{s\in[t,T],|u|\le M+2}|\xi_{s,\eps,M}^{t,x} (u)|>L\Big]=0.
\end{equation}
\end{lemma}

\subsubsection*{Proof of Lemma \ref{lm012509}}
{\em Proof of \eqref{032509} and \eqref{042509}.}
Using the Sobolev embedding, we can
estimate 
\begin{align}
\label{071409}
\sup_{s\in[t,T],|u|\le
M+2}\left|\Theta_{\eps,M}(s,u)\right|\le \sup_{(s,y,v)\in
\tilde
    D_{M+2}}|\Theta(s,y,v)|
\le C_M\left\{\|\Theta\|_{L^p(\tilde
    D_{M+2})}+\|\nabla\Theta\|_{L^p(\tilde D_{M+2})}\right\},
\end{align}
with  $p>d+2$,  constant $C_M>0$ independent of $\eps$  and 
\[
\tilde D_M:=[t,T]\times  K_M.
\]
Taking the expectation in 
both sides of  \eqref{071409}, we obtain that, for each $N,M\ge1$:
\begin{equation}
\label{061409}
\lim_{\eps\to 0}\bbE\Big[\eps \sup_{s\in[t,T],\,|u|\le M+2}|\Theta_{\eps,M}(s,u)|\Big]=0,
\end{equation}
and \eqref{032509} follows.

A similar argument shows that
\begin{align}
\label{051609}
&
\limsup_{\eps\to 0}\bbE\Big[\sup_{s\in[t,T],\,|u|\le M+2}|\nabla_x\Theta_{\eps,M}(s,u)\cdot
v(\eta^\eps_{t,x}(s))|\Big]<+\infty,\\
&
\limsup_{\eps\to 0}\bbE\Big[\sup_{s\in[t,T],\,|u|\le M+2}
|\Theta'_{\eps,M|u}(s,u) f_{\eps,M}(s,u)|\Big]<+\infty.\nonumber
\end{align}
These  estimates imply 
\eqref{042509}. 


{\em Proof of  \eqref{042509a}.} 
We start with the following ``finite-rank'' approximation. 
\begin{lemma}\label{l.finiterank}
 {Let $M>1$, $m,m_1\ge0$ and  $f:\tilde D_M\times{\cal E}\to\bbR$ be such
that 
$$
\max_{|k|\le m_1}{\rm esssup}_{w\in{\cal E}}\|D^k f\|_{C^m(\tilde D_M)}<+\infty.
$$
Then,  
for any $\delta>0$, there exist $\varphi_1,\ldots,\varphi_N\in
C^{m}(D_T)$ and $\Phi_1,\ldots,\Phi_N\in {\cal W}_{m_1,\infty}$ (cf \eqref{wkp}) such that
\begin{equation}
\label{012709}
\max_{|k|\le m_1}{\rm esssup}_{w\in{\cal E}}\|D^kf(\cdot,w)-
D^k \tilde  f(\cdot,w)\|_{C^{m}(\tilde D_M)}<\delta,
\end{equation}}
where
\begin{equation}
\label{022709}
\tilde f(s,y,u,w):=\sum_{j=1}^N\varphi_j(s,y,u)\Phi_j(w),\quad (s,y,u,w)\in D_T\times{\cal E}.
\end{equation}
\end{lemma}
\begin{proof}
To simplify the presentation, we assume that $f$ does not depend on $y$ and
$k=0$. Let us partition the rectangle $[t,T]\times [-M,M]$  
using the grid points 
\[
t=t_0<\ldots<t_n=T,~-M=m_0<\ldots<m_n=M,
~t_i-t_{i-1}=\frac{T-t}{n},~m_i-m_{i-1}=\farc{2M}n,~\!i=1,\ldots,n.
\]
Let
$\Delta_{i,i'}:=[t_{i-1},t_i]\times [m_{i'-1},m_{i'}]$ and 
$$
m_{i'}^*:=\frac{m_{i'}+m_{i'-1}}{2},\quad t_i^*:=\frac{t_i+t_{i-1}}{2},\quad i,i'=1,\ldots,n,
$$
and $\Delta^0_{i,i'}$ be an open neighborhood of
$\Delta_{i,i'}$ contained in  
$$
\left[t_{i-1}-\frac{T-t}{2n},t_i+\frac{T-t}{2n}\right]\times \left[m_{i'-1}-\frac{M}{n},m_{i'}+\frac{M}{n}\right].
$$ 
Let
$\phi_{i,i'}:\bbR^2\to[0,1]$ be a smooth partition of unity on
$[t,T]\times [-M,M]$ subordinated to the
open covering $\left(\Delta^0_{i,i'}\right)$  of $[t,T]\times
[-M,M]$. We may assume that $\phi_{i,i'}\equiv 1$ in some neighborhood
of $(t_i^*,m_{i'}^*)$.

 Since
the function $(s,u)\mapsto f(s,u,\cdot)$ is uniformly continuous from
$[t,T]\times [-M,M]$ to $L^\infty(\pi)$, we can choose $n$ sufficiently large so that
$|f(s,u)-f(s',u')|<\delta$ for $(s,u),(s',u')\in \Delta_{i,i'}^0$. Let
$$
\tilde f(s,y,u,w):=\sum_{i,i'}\phi_{i,i'}(s,u)f(t_i^*,m_{i'}^*,w),\quad (s,u,w)\in [t,T]\times [-M,M]\times{\cal E}.
$$
One can easily verify that then \eqref{012709} holds with $m=0$.
\end{proof}

We go back to the proof of Lemma~\ref{lm012509}. 
Choose arbitrary $\delta,\eta>0$ and choose $\tilde f$, of the form~\eqref{022709}, so that it 
satisfies \eqref{012709}, with  $\delta$, $f$ replaced by
$\delta\eta$ and $f_M$, respectively, and $m=d+2$.
Let $\Theta_j$ be the solution of 
\[
-{\cal L}\Theta_j=\Phi_j,
\]
with $\Phi_j$ constructed in Lemma~\ref{l.finiterank}, then
$
\tilde\Theta^{(0)}(s,y,u,w):=\sum_{j=1}^N\varphi_j(s,y,u)\Theta_j(w)
$
satisfies
$$
\sum_{\ell=0}^{d+2}\sum_{|k|=\ell}\sup_{(s,y,u)\in \tilde D_{M+2}}{\cal E}_L\left(\nabla^k\tilde\Theta^{(0)}(s,y,u)-\nabla^k\Theta_M(s,y,u)\right)\le
C\frac{(\delta\eta)^2}{\al_*}.
$$
The constant $C$ depends only on $d$. Thus, approximating $\Theta_j$, 
if needed, we can find  
$\tilde\Theta_j\in L^\infty(\pi)$ such that ${\cal D}\tilde\Theta_j\in
L^\infty(\pi,H)$, and 
\begin{equation}
\label{052709}
\sum_{\ell=0}^{d+2}\sum_{|k|=\ell}\sup_{(s,y,u)\in 
\tilde D_{M+2}}{\cal E}_L(\nabla^k\tilde\Theta(s,y,u)-\nabla^k\Theta_M(s,y,u))\le
C\frac{(\delta\eta)^2}{\al_*},
\end{equation}
with
\begin{equation}
\label{Phi}
\tilde\Theta(s,y,u,w):=\sum_{j=1}^N\varphi_j(s,y,u)\tilde\Theta_j(w).
\end{equation}
Define
\begin{equation}
\label{082709}
\tilde{\mathfrak
      S}^{t,x}_{\eps}(s,u):=\int_t^s
 \langle  \tau_{-X^{t,x}_\eps(\si)/\eps}
 {\cal D}\tilde\Theta_{\eps}\left(\si,u\right),dB^\eps_\si\rangle_H,
\end{equation}
where 
\[
\tilde{\Theta}_\eps(s,u)=\tilde{\Theta}(s,X_\eps^{t,x}(s),U_{\eps}^{t,x,u}(s),\eta^\eps_{t,x}(s)).
\]
By the Sobolev embedding, there exists a constant $C>0$ such that
\begin{align}
\label{042709}
&
\sup_{s\in[t,T],|u|\le M+2}|\tilde{\mathfrak
      S}^{t,x}_{\eps}(s,u)-{\mathfrak
      S}^{t,x}_{\eps,M,2}(s,u)|\\
&
\le C\int_{-M-2}^{M+2} \Big\{\sup_{s\in[t,T]}|\tilde{\mathfrak
      S}^{t,x}_{\eps}(s,u)-{\mathfrak
      S}^{t,x}_{\eps,M,2}(s,u)|+\sup_{s\in[t,T]}|\partial_u\tilde{\mathfrak
      S}^{t,x}_{\eps}(s,u)-\partial_u{\mathfrak
      S}^{t,x}_{\eps,M,2}(s,u)|\Big\}du.\nonumber
\end{align}
Applying expectation to both sides of \eqref{042709} and using Doob's
inequality, we obtain
\begin{align}
\label{042709a}
&
\bbE\Big[\sup_{s\in[t,T],|u|\le M+2}\left|\tilde{\mathfrak
      S}^{t,x}_{\eps}(s,u)-{\mathfrak
      S}^{t,x}_{\eps,M,2}(s,u)\right|\Big]\\
&
\le C\int_{-M-2}^{M+2}
du\, \bbE\Big\{\int_t^T\left[\left|\left({\cal D}\tilde\Theta_{\eps}-{\cal D}\Theta_{\eps,M}\right)(s,u)\right|_H^2+\left|\left(\partial_u {\cal D}\tilde\Theta_{\eps}-\partial_u {\cal D}\Theta_{\eps,M}\right)(s,u)\right|_H^2\right]ds\Big\}^{1/2}\nonumber\\
&
\le 2C(M+2)\bbE\Big\{\int_t^T\sup_{(y,v)\in K_{M+2}}\Big[\left|({\cal D}\tilde\Theta-{\cal D}\Theta_{M})(s,y,v,\eta^\eps_{t,x}(s))\right|_H^2\nonumber\\
&+\left|\left(\partial_v {\cal D}\tilde\Theta-\partial_v {\cal D}\Theta_{M}\right)(s,y,v,\eta^\eps_{t,x}(s))\right|_H^2\Big]ds\Big\}^{1/2}.\nonumber
\end{align}
Using again the Sobolev estimate, this time to estimate the supremum of  
\[
\sup_{(y,v)\in K_{M+2}}|({\cal D}\tilde\Theta-{\cal
      D}\Theta_{M})(s,y,v ,\eta^\eps_{t,x}(s))|_H^2,
\]
we conclude
that there exist constants $C,C'>0$ such that 
\[
\sup_{(y,v)\in 
      K_{M+2}}|({\cal D}\tilde\Theta-{\cal
      D}\Theta_{M})(s,y,v ,\eta^\eps_{t,x}(s))|_H^2
\le C'\sum_{k=0}^{d+1}\int_{K_{M+2}}|\nabla^k({\cal D}\tilde\Theta-{\cal
      D}\Theta_{M})(s,y,v ,\eta^\eps_{t,x}(s))|_H^2dydv. 
\]
A similar estimate holds for $ |\partial_v({\cal D}\tilde\Theta-{\cal
      D}\Theta_{M})(s,y,v ,\eta^\eps_{t,x}(s))|_H^2$, leading to 
\begin{align}
\label{042709b}
&
\bbE\Big[\sup_{s\in[t,T],|u|\le M+2}|\tilde{\mathfrak
      S}^{t,x}_{\eps}(s,u)-{\mathfrak
      S}^{t,x}_{\eps,M,2}(s,u)|\Big]\\
&
\le 2C(M+2)\Big\{\sum_{\ell=0}^{d+1}\sum_{|k|=\ell}\int_{\cal
    E}d\pi\Big[\int_{\tilde D_{M+2}}|\nabla^k({\cal D}\tilde\Theta-{\cal D}\Theta_{M})(s,y,v)|_H^2dsdydv\Big]\Big\}^{1/2},\nonumber
\end{align}
with a constant $C>0$ depending only on $M$ and $d$.
By virtue of \eqref{052709}, we conclude that
\begin{align}
\label{042709c}
&
\bbE\Big[\sup_{s\in[t,T],|u|\le M+2}|\tilde{\mathfrak
      S}^{t,x}_{\eps}(s,u)-{\mathfrak
      S}^{t,x}_{\eps,M,2}(s,u)|\Big]
\le C\delta\eta \al_*^{-1/2}.
\end{align}
It follows from the Chebyshev inequality that there exists 
$C>0$ depending only on $M$, $d$ and $\al_*$ such that for
any $\delta,\eta>0$ we can find $\tilde\Theta$ of the form
\eqref{Phi}
such that
\begin{align}
\label{042709d}
\limsup_{\eps\to0}\bbP\Big[\sup_{s\in[t,T],|u|\le M+2}|\tilde{\mathfrak
      S}^{t,x}_{\eps}(s,u)-{\mathfrak
      S}^{t,x}_{\eps,M,2}(s,u)|>\eta\Big]
\le C\delta.
\end{align}
The above, in particular, implies   \eqref{042509}, if we 
prove that for any  $\tilde{\mathfrak
      S}^{t,x}_{\eps}(s,u)$ of the form \eqref{082709} 
and any $\eta>0$  we have
\begin{equation}
\label{092709}
\lim_{\delta\to0}\limsup_{\eps\to0}\bbP[\tilde Z_{\eta,\delta,\eps}]=0,
\end{equation}
with
$$
\tilde Z_{\eta,\delta,\eps}:=\Big[\mathop{\sup_{t\le
      s<s'\le T,|u|\le M+2}}_{s'-s<\delta}\left|\tilde{\mathfrak
      S}^{t,x}_{\eps}(s',u)-\tilde{\mathfrak
      S}^{t,x}_{\eps}(s,u)\right|>\eta\Big].
$$
To this end,  
we invoke the Sobolev inequality, as in
\eqref{042709} and \eqref{042709a}. This, together with
Burkolder-Davis-Gundy inequality, implies
\begin{align*}
\bbE\Big[\sup_{|u|\le M+2}[\tilde{\mathfrak
      S}^{t,x}_{\eps}(s',u)-\tilde{\mathfrak
      S}^{t,x}_{\eps}(s,u)]^4\Big]
\le C\int_{-M-2}^{M+2}
du\bbE\Big\{\int_s^{s'}\left(|\mathcal{D}\tilde\Theta_{\eps}(\si,u)|^2_H+|\partial_u\mathcal{D}\tilde\Theta_{\eps}(\si,u)|^2_H\right)d\si\Big\}^{2},
\end{align*}
for any $s'>s$. Since $N$ is finite, we also have
\begin{equation}
\label{bound}
{\rm esssup}_{w\in{\cal E}}\sup_{(s,y,u)\in\tilde D_{M+2}}\big(
|\mathcal{D}\tilde\Theta(s,y,u)|_H+|\partial_u\mathcal{D}\tilde\Theta(s,y,u)|_H\big)<+\infty,
\end{equation}
whence
\begin{align*}
A_*:=\bbE\Big[\int_t^T\int_t^T(s-s')^{-5/2}\sup_{|u|\le M+2}\big[\tilde{\mathfrak
      S}^{t,x}_{\eps}(s',u)-\tilde{\mathfrak
      S}^{t,x}_{\eps}(s,u)\big]^4dsds'\Big]<+\infty.
\end{align*}
Let $\rho>0$ be arbitrary. By virtue of Chebyshev inequality we obtain that 
$
\bbP\left[Z_{\rho,\eps}\right]<\rho$ for all $\eps\in(0,1],
$
with
$$
Z_{\rho,\eps}:=\left[\int_t^T\int_t^T(s-s')^{-5/2}\sup_{|u|\le M+2}\left[\tilde{\mathfrak
      S}^{t,x}_{\eps}(s',u)-\tilde{\mathfrak
      S}^{t,x}_{\eps}(s,u)\right]^4dsds'>\frac{A_*}{\rho}\right].
$$
The Garcia-Rodemich-Rumsey estimate,
see Theorem 2.1.3 of \cite{stroock-varadhan},   implies
that given $\rho>0$ there exists $\delta_0>0$ such that
$$
\mathop{\sup_{t\le
      s<s'\le T,|u|\le M+2}}_{s'-s<\delta}\left|\tilde{\mathfrak
      S}^{t,x}_{\eps}(s',u)-\tilde{\mathfrak
      S}^{t,x}_{\eps}(s,u)\right|\le
  40\sqrt{2}\left(\frac{A_*}{\rho}\right)^{1/4}\delta^{1/8}<\eta,\quad\eps\in(0,1], \,\delta\in(0,\delta_0)
$$
 on the event $Z_{\rho.\eps}^c$.
Hence, for any $\rho>0$ there exists $\delta_0>0$ such that 
$
\tilde Z_{\eta,\delta,\eps}\subset Z_{\rho,\eps}
$
for all $\eps\in(0,1]$ and $\delta\in(0,\delta_0)$, which yields
$$
\limsup_{\delta\to0}\limsup_{\eps\to0}\bbP\left[\tilde
  Z_{\eta,\delta,\eps}\right]\le \rho.
$$
This in turn implies \eqref{092709}, as $\rho>0$ can be made
arbitrarily small. This ends the proof of  \eqref{042509a}
and thus that of Lemma~\ref{lm012509}.\qed

\subsubsection*{Proof of Lemma \ref{lm022509}}

Arguing as in the proof of \eqref{042509}, we can show  that 
\begin{equation}
\label{012809a}
\limsup_{\eps\to 0}\bbE\Big[\sup_{|u|\le M+2,\,s\in[t,T]}\big\{|
      \Theta'_{\eps,M|u}(s,u)|+| \partial_s\Theta'_{\eps,M|u}(s,u)|
+| \nabla_x\Theta'_{\eps,M|u}(s,u)|\big\}\Big]<+\infty.
\end{equation}
Given $L> 1$, define the event 
$$
E_{L,\eps}:=\Big[\sup_{|u|\le M+2,\,s\in[t,T]}\left\{|
      \Theta'_{\eps,M|u}(s,u)|+| \partial_s\Theta'_{\eps,M|u}(s,u)|
+| \nabla_x\Theta'_{\eps,M|u}(s,u)|\right\}>L\Big].
$$
It follows from (\ref{012809a}) that for an arbitrary $\delta>0$, 
there exists $L>1$ so that
\begin{equation}
\label{EMN}
\bbP[E_{L,\eps}]<\delta,\quad \hbox{ for all $\eps\in(0,(2L)^{-1})$}.
\end{equation}
Next, on the event $E_{L,\eps}^c$, let us define 
\[
\tilde \xi^{t,x}_{s,\eps,M}(u) :=\Theta_{1,\eps,M}(s)\xi^{t,x}_{s,\eps,M}(u), \    \  \Theta_{1,\eps,M}(s):=1+\eps \Theta'_{\eps,M|u}(s).
\]  
Recall that 
\[
\xi_{s,\eps,M}^{t,x}(u)=1+\frac{1}{\eps}\int_t^s f'_{\eps,M|u}(\sigma)\xi^{t,x}_{\sigma,\eps,M}(u)d\sigma,
\]
using the It\^o formula and the fact that $-\mathcal{L} \Theta_{M|u}'= f_{M|u}'$,
we can write
\begin{align}
\label{011906e}
&
\tilde \xi^{t,x}_{s,\eps,M}(u) =1+\eps \Theta'_{\eps,M|u}\left(t,x,u\right)+\int_t^s
\tilde \xi^{t,x}_{\si,\eps,M}(u) \tilde \ga_\eps(d\si)
+\int_t^s\tilde \xi^{t,x}_{\si,\eps,M}(u)\tilde\al_\eps(\si)d\si,
\end{align}
with
\begin{align*}
&\tilde\ga_\eps(d\si):=\Theta_{1,\eps,M}^{-1}(\si)
\langle  \tau_{-X^{t,x}_\eps(\si)/\eps}{\cal D}\Theta'_{\eps,M|u}(\si),dB^\eps_\si\rangle_H,\\
 &
 \tilde\al_\eps(\si):=\Theta_{1,\eps,M}^{-1}(\si)\left\{ \eps\partial_{\si}\Theta'_{\eps,M|u}(\si) 
+
\nabla_x\Theta_{\eps,M|u}'\left(\si\right)\cdot
v(\eta^\eps_{t,x}(\si))+ {\Theta''_{\eps,M|u}\left(\si\right)} {f_{\eps,M}(\sigma)}+ {\Theta'_{\eps,M|u}\left(\si\right)f_{\eps,M|u}'(\sigma)}\right\}.
\end{align*}
The above allow us to write
\begin{align}
\label{011906f}
\tilde \xi^{t,x}_{s,\eps,M}(u) =(1+\eps
  \Theta'_{\eps,M|u}(t,x,u))\exp\big\{\tilde {\cal Z}_\eps(s,u)\big\},
\end{align}
with 
\begin{align}
\label{020307c}
\tilde
{\cal Z}_\eps(s,u):=\int_t^s\big\{\tilde \al_\eps(\si)-\frac{1}{2}\langle\tilde
  \ga_\eps\rangle_\si\big\}d\si+\int_t^s\tilde \ga_\eps(d\si),
\end{align}
and
$$
\langle\tilde
  \ga_\eps\rangle_\si:=\Theta_{1,\eps,M}^{-2}(\si)
\langle A{\cal D}\Theta'_{\eps,M|u}\left(\si\right),{\cal D}\Theta'_{\eps,M|u}\left(\si\right)\rangle_H.
$$
Using the Sobolev embedding argument, as in \eqref{042709a} -- \eqref{042709b}, we see that
\begin{equation}
\label{012809}
\limsup_{\eps\to 0}\int_t^T\bbE\Big[
\sup_{|u|\le M+2}\langle A{\cal D}\Theta'_{\eps,M|u}(\si),{\cal D}\Theta'_{\eps,M|u}(\si)\rangle_H\Big]d\si<+\infty.
\end{equation}
Combining \eqref{012809a} and \eqref{012809}, we
conclude that for any
$\delta>0$ there exists $L>1$ such that
$$
\limsup_{\eps\to0}\bbP\Big[\sup_{|u|\le M+2,\,s\in[t,T]}|\tilde{\cal
Z}_\eps(s,u)|>L, \,E_{L,\eps}^c\Big]<\delta.
$$
This together with \eqref{EMN} implies \eqref{011609d}. The proof of
Lemma~\ref{lm022509} is complete. \qed

\subsection{The weak convergence of $u_{\eps,M}(t,x)$}

The goal in this section is to show that not only 
the ``forward'' processes $({\mathfrak S}_{\eps,M}^{t,x}(\cdot),
X_{\eps}^{t,x}(\cdot))$ converge as~$\eps\to0$ but also 
the ``inverse'' processes, i.e., $u_{\eps,M}(t,x)$ and $u_{M}(t,x)$ given by
\eqref{010911M} and \eqref{010911MM}, are close in law. Given $t\in[0,T]$, define 
\[
{\mathfrak X}:=C([t,T]\times \bbR;\bbR)\times
C([t,T];\bbR^d).
\]
\begin{proposition}
\label{prop013009bis}
The random elements $({\mathfrak
    S}_{\eps,M}^{t,x}(\cdot),X_{\eps}^{t,x}(\cdot),u_{\eps,M}(t,x))$
converge, as $\eps\to 0$, in law over~$C([t,T]\times \bbR;\bbR)\times
C([t,T];\bbR^d)\times\bbR$, 
to $({\mathfrak
    S}_{M}^{t,x}(\cdot),X^{t,x}(\cdot), u_{M}(t,x))$.
\end{proposition}
\proof
Given any $N>1$, denote by 
${\cal CM}_+(N)$ the $G_{\delta}$ subset of ${\mathfrak X}$ 
that consists of all $(S(\cdot),x(\cdot))\in {\mathfrak X}$ such that $S$ is a continuous function
$S:[t,T]\times \bbR\to\bbR$,  strictly
increasing in the second variable, and $S(s,u)\equiv u$ for all
$s\in[t,T]$, $|u|\ge N$. 
Also, for a   ${\mathfrak X}$-valued random element $({\mathfrak S},X)$, 
denote by  ${\cal L}({\mathfrak S},X)$
its law.

Recall that $u_0$, the terminal condition of \eqref{advec-11-semi2}, belongs to $C_0^\infty(\bbR^d)$, so there exists $K>1$ such that the
range of $u_0(\cdot)$  is contained in $[-K,K]$. Then, we have 
\[
{\mathfrak s}_{T,\eps,M}^{t,x}(u)\equiv u,
\hbox{ for $|u|\ge M_*:=\max[M+2,K]$,}
\] 
and, the laws ${\cal L}_\eps:={\cal L}({\mathfrak
     S}_{\eps,M}^{t,x}(\cdot),X_\eps^{t,x}(\cdot))$ and ${\cal L}:={\cal L}({\mathfrak
    S}_{M}^{t,x}(\cdot),X^{t,x}(\cdot))$
are  supported in  ${\cal CM}_+(M_*)$.

Let
${\cal H}:{\mathfrak X}\to\bbR$ be given by
$$
{\cal H}(S,X):=S^{-1}(T,u_0(X(T))),\quad (S,X)\in {\cal CM}_+(M_*).
$$
Here $S^{-1}(T,\cdot)$ is the inverse of
$S$ 
in the second variable.
Outside of $ {\cal CM}_+(M_*)$, we can define ${\cal H}$ arbitrarily,
for instance, as a  constant. The mapping is
measurable, bounded and continuous on~${\cal CM}_+(M_*)$.

 {We know from Proposition~\ref{prop012509} that
$({\mathfrak
    S}_{\eps_n,M}^{t,x}(\cdot),X_{\eps_n}^{t,x}(\cdot))$
converge in law to $({\mathfrak
    S}_{M}^{t,x}(\cdot),X^{t,x}(\cdot))$, for any~$\eps_n\to 0$.  
We also have
\begin{equation}\label{uM}
u_{\eps_n,M}(t,x)={\cal H}({\mathfrak
    S}_{\eps_n,M}^{t,x}(\cdot),X_{\eps_n}^{t,x}(\cdot))\quad\mbox{and}\quad u_{M}(t,x)={\cal H}({\mathfrak
    S}_{M}^{t,x}(\cdot),X^{t,x}(\cdot)).
\end{equation}
By the continuous mapping theorem, see Theorem 2.7, p. 21 of \cite{billingsley},
$u_{\eps_n,M}(t,x)$ converges in law to $u_{M}(t,x)$, as~$n\to+\infty$. }\qed

\subsection{Proof of part (i) of Theorem~\ref{thm011110}}
\label{sec8.1.2}

Let us denote by ${\mathfrak Q}_{\eps}$, ${\mathfrak Q}_{\eps,M}$, 
${\mathfrak Q}_M$,
${\mathfrak Q}$ the respective  laws of the random elements   
\[
\hbox{${\mathfrak Y}_{\eps}:=\left({\mathfrak
    S}_{\eps}^{t,x}(\cdot),X_{\eps}^{t,x}(\cdot)\right)$,  ${\mathfrak Y}_{\eps,M}:=({\mathfrak
    S}_{\eps,M}^{t,x}(\cdot),X_{\eps}^{t,x}(\cdot))$,  ${\mathfrak Y}_{M}:=({\mathfrak
    S}_{M}^{t,x}(\cdot),X^{t,x}(\cdot))$,   ${\mathfrak Y}:=({\mathfrak
    S}^{t,x}(\cdot),X^{t,x}(\cdot))$},
\]
over ${\mathfrak X}$. By Proposition~\ref{prop012509}, we have
 ${\mathfrak Q}_{\eps,M}\Longrightarrow{\mathfrak Q}_M$, as $\eps\to 0$ for
 each $M>1$. A standard argument based on local uniform convergence of coefficients
 $b_M$ to $b$ and $\tilde c_M$ to $\tilde{c}$, see e.g. Section 9.2.6, pp.~528-529
of \cite{JS}, implies that also ${\mathfrak Q}_{M}\Longrightarrow{\mathfrak Q}$, as $M\to+\infty$.
 

Given $N,M>1$, recall that
\[
K_M=\{(y,u):\,|y|< M,\,|u|< M\},
\]
we define  
$T_{M,N}:{\mathfrak X}\to[0,T+1]$ by
$$
T_{M,N}({\mathfrak S},X):=\inf\left\{s\in[t,T]:\,(X(s),{\mathfrak S}(s,u))\not\in
  K_M\quad\mbox{for some }|u|\le N\right\}.
$$
We adopt the convention that $T_{M,N}({\mathfrak S},X):=T+1$, if the set, over which
the infimum is taken, is empty.

Let $\left({\cal F}_{N,s}\right)_{N>1,s\in[t,T]}$ be the family
of $\si$-algebras 
 generated by $\left({\mathfrak S}(\si,u),X(\si)\right)$, 
 with~$|u|\le N$, and~$\si\in[t,s]$. Note that
$$
{\mathfrak Q}_{\eps}[A,T_{M,N}>T]={\mathfrak
  Q}_{\eps,M}[A,T_{M,N}>T]\quad\mbox{and}\quad {\mathfrak Q}_{M}[A,T_{M,N}>T]={\mathfrak
  Q}[A,T_{M,N}>T]
$$
for any $A\in {\cal F}_{N,T}$, $M,N>1$, $\eps>0$.

The following  estimate on the random time $T_{M,N}$ is crucial in removing the
truncation. 
\begin{proposition}
\label{prop013009}
For any $\rho>0$ and $N>1$, there exists $M>1$ such that
\begin{equation}
\label{013009}
\sup_{M'\geq M}{\mathfrak
  Q}[T_{M',N}\le T]<\rho.
\end{equation}
\end{proposition}
\begin{proof}

Since $
\xi_{s}^{t,x}(u)=\partial_uU^{t,x,u}(s)>0$, 
we have 
\[
\hbox{$U^{t,x,N}(s)\ge U^{t,x,u}(s)\ge U^{t,x,-N}(s)$ for $|u|\le N$}.
\]
It suffices to show
that for any $\rho>0$  and $N>1$, there exists $M>1$ such that
\begin{equation}
\label{033009}
\bbP\Big[\sup_{s\in[t,T]}|X^{t,x}(s)|\ge M\Big]+
\bbP\Big[\sup_{s\in[t,T]}|U^{t,x,N}(s)|\ge M\Big]+\bbP\Big[\sup_{s\in[t,T]}|U^{t,x,-N}(s)|\ge M\Big]<\rho.
\end{equation}
Since $(X^{t,x}(s))_{s\in[t,T]}$ is a Brownian motion,
there exist  
$C_1,C_2>0$ such that 
$$
\bbP\Big[\sup_{s\in[t,T]}|X^{t,x}(s)|\ge
  M\Big]\le C_1\exp\Big\{-C_2\frac{M^2}{T-t}\Big\}<\frac{\rho}{3},
$$
provided that $M>1$ is sufficiently large. Similarly, we deduce from \eqref{020307} that
\begin{align*}
\bbP\Big[\sup_{s\in[t,T]}|U^{t,x,N}(s)|\ge M\Big]
\le C_1\exp\Big\{-C_2\frac{(M-N-\|b\|_\infty
  (T-t))^2}{T-t}\Big\}<\frac{\rho}{3},
\end{align*}
for
 $M-N$ sufficiently large. A similar estimate holds for
 $\bbP\Big[\sup_{s\in[t,T]}|U^{t,x,-N}(s)|\ge M\Big]$, which completes
 the proof of the proposition.
\end{proof}
\bigskip

Now we can  finish the proof of part (i) of Theorem
\ref{thm011110}.
Fix any $\rho>0$, $N>1$ and~$F\in C_b({\mathfrak X})$. Since $\{T_{M',N}\le T\}$ is a closed
subset of ${\mathfrak X}$ for any~$N,M>1$, we have
\begin{equation}
\label{023009}
\limsup_{M'\to+\infty}{\mathfrak
  Q}_{M'}[T_{M,N}\le T]\le {\mathfrak
  Q}[T_{M,N}\le T].
\end{equation}
Let $M>0$ be chosen as in the statement of Proposition
\ref{prop013009}. We can write then, for any $M'\ge M$
\begin{align*}
&
\Big|\int_{\mathfrak X}Fd {\mathfrak Q}_{\eps}-\int_{\mathfrak X}Fd {\mathfrak
    Q}\Big|\le \Big|\int_{T_{M,N}>T}Fd {\mathfrak Q}_{\eps}-\int_{T_{M,N}>T}Fd {\mathfrak
    Q}\Big|+\|F\|_{\infty}({\mathfrak Q}_{\eps}[{T_{M,N}\le T}]+{\mathfrak
  Q}[{T_{M,N}\le T}])\\
&
\le {\Big|\int_{T_{M,N}>T}Fd {\mathfrak Q}_{\eps,M'}-\int_{T_{M,N}>T}Fd {\mathfrak
    Q}_{M'}\Big|}+\|F\|_{\infty}({{\mathfrak Q}_{\eps,M'}[{T_{M,N}\le
  T}]}+\rho)\\
&
\le \Big|\int_{\mathfrak X}Fd {\mathfrak Q}_{\eps,M'}-\int_{\mathfrak X}Fd {\mathfrak
    Q}_{M'}\Big|+\|F\|_{\infty}(2{\mathfrak Q}_{\eps,M'}[{T_{M,N}\le T}]+{\mathfrak
  Q}_{M'}[{T_{M,N}\le T}]+\rho)
\end{align*}
Passing to the limit, first as $\eps\to 0$, 
and then $M'\to+\infty$, using \eqref{013009},  \eqref{023009},
and the convergence~${\mathfrak Q}_{\eps,M'}\Longrightarrow{\mathfrak Q}_{M'}$, 
we deduce that
\begin{align*}
\limsup_{\eps\to 0}\Big|\int_{\mathfrak X}Fd {\mathfrak Q}_{\eps}-\int_{\mathfrak X}Fd {\mathfrak
    Q}\Big|
\le4\rho\|F\|_{\infty}.
\end{align*}
Since $\rho>0$ can be  chosen arbitrarily small, we conclude that
$$
\lim_{\eps\to0}\Big|\int_{\mathfrak X}Fd {\mathfrak Q}_{\eps}-\int_{\mathfrak X}Fd {\mathfrak
    Q}\Big|=0,
$$
which concludes the proof of part (i) of Theorem \ref{thm011110}.


\subsection{Proof of part (ii) of Theorem \ref{thm011110}}

By the tightness of the laws of $\left({\frak S}_\eps^{t,x}(\cdot),X_\eps^{t,x}(\cdot)\right)$, 
for any $\rho>0$, there exists $N>1$ such that 
$$
\bbP[A_{N}]<\frac{\rho}{3},\quad \bbP[A_{N,\eps}]<\frac{\rho}{3},\quad \eps\in(0,1],
$$  
where
{
$$
A_{N,\eps}:=\left[{\frak s}_{T,\eps}^{t,x}(N)\leq \|u_0\|_\infty\quad\mbox{or}\quad {\frak s}_{T,\eps}^{t,x}(-N)\geq -\|u_0\|_\infty\right],
$$
}
and $A_N$ is defined analogously with ${\frak s}_{T,\eps}^{t,x}$ replaced by ${\frak s}_{T}^{t,x}$.

According to Proposition \ref{prop013009}, we can find $M>1$ such that $\bbP[B_{N,M}]<\rho/3$, with
$$
B_{N,M}:=\left[\sup_{(s,u)\in[t,T]\times[-N,N]}|U^{t,x,u}(s)|\ge M,\sup_{s\in[t,T]}|X^{t,x}(s)|\ge M\,\right].
$$
Using part (i) of Theorem \ref{thm011110}, we conclude that $\bbP[B_{N,M,\eps}]<\rho/3$, $\eps\in(0,1]$,
where $B_{N,M,\eps}$ is defined analogously for $\left({\frak S}_\eps^{t,x}(\cdot),X_\eps^{t,x}(\cdot)\right)$.
Thanks to equality \eqref{uM} and the fact that 
\[
{\frak s}_T^{t,x}(u(t,x))=u_0(X^{t,x}(T)), \   \  {\frak s}_{T,\eps}^{t,x}(u_\eps(t,x))=u_0(X^{t,x}_\eps(T)),
\]
we conclude that
$$
u_{M'}(t,x)=u(t,x),\quad u_{\eps,M'}(t,x)=u_\eps(t,x) ,\quad M'\ge M
$$ outside $A_N\cup B_{N,M}\cup A_{N,\eps}\cup B_{N,M,\eps}$. Using the already proved  convergence in law of $ u_{\eps,M'}(t,x)$ to $u_{M'}(t,x)$, as $\eps\to 0$,
we conclude from the above that $u_\eps(t,x)$ converges in law to $u(t,x)$, $\eps\to 0$ by the argument  presented in Section \ref{sec8.1.2}.

 {The convergence of the multi-point statistics
follows  by essentially the same argument as in the proof of
Theorem~\ref{thm-semilin}. For $N$ distict points
$x_1,\ldots, x_N\in\bbR^d$, $u_1,\ldots,u_N\in\bbR$,  the respective
processes 
$$
\left((U_\eps^{t,x_1,u_1}(s),X_\eps^{t,x_1}(s)),\ldots,
  (U_\eps^{t,x_N,u_N}(s),X_\eps^{t,x_N}(s))\right)_{s\ge t} $$
 converge
in distribution to
$$
\left((U_1^{t,x_1,u_1}(s),X_1^{t,x_1}(s)),\ldots,
  (U_N^{t,x_N,u_N}(s),X_N^{t,x_N}(s))\right)_{s\ge t} ,
$$ where
$\left((U_j^{t,x_j,u_j}(s),X_j^{t,x_j}(s))\right)_{s\ge t}$,
$j=1,\ldots,N$ are independent copies of solutions of \eqref{020307}.
This implies that the respective ${\mathfrak
    s}_{T}^{t,x_j}(\cdot)$, $j=1,\ldots,N$ are independent and, as a
  result, allows us to infer that $\mathscr{U}^{(j)}(t,x_j)$
  determined by the corresponding equations \eqref{011001} are also independent.}\qed

\bibliographystyle{amsalpha}

\end{document}